\documentclass{amsart}
\usepackage{amscd,amssymb,amsmath,subfigure,hyperref,mathrsfs }
\usepackage[arrow,matrix,graph,frame,poly,arc,tips]{xy}

\theoremstyle{plain}
\newtheorem{thm}[subsection]{Theorem}
\newtheorem{prop}[subsection]{Proposition}
\newtheorem{lem}[subsection]{Lemma}

\theoremstyle{definition}

\theoremstyle{remark}

\newcommand{\mc}[1]{\mathcal{#1}}
\newcommand{\mbf}[1]{\mathbf{#1}}
\newcommand{\mfr}[1]{\mathfrak{#1}}
\newcommand{\mr}[1]{\mathrm{#1}}
\newcommand{\mbb}[1]{\mathbb{#1}}

\newcommand{\R}{\mathbb{R}}
\newcommand{\C}{\mathbb{C}}

\newcommand{\Z}{\mathbb{Z}}
\newcommand{\Q}{\mathbb{Q}}

\newcommand{\Proj}[1]{{\bf P}^{#1}}
\newcommand{\Aff}[1]{{\bf A}^{#1}}

\newcommand{\Sg}[1]{{\mathfrak H}_{#1}}

\newcommand{\Sym}[2]{\mr {Sp}_{#1} (#2)}
\newcommand{\sym}[2]{\mr {Sp} (#1,\, #2)}

\newcommand{\SO}[2]{{\mr {SO}} ( #1, \, #2)}

\newcommand{\kmphi}{\boldsymbol {\varphi}^{+}}

\newcommand{\wdg}[1]{\mathop{\bigwedge ^{#1}}}

\newcommand{\coh}[3]{{\mr H}^{#1} (#2 ,\ #3)}
\newcommand{\cohb}[3]{{\mr H}^{#1} (#2 ;\ #3)}

\newcommand{\mat}[3]{\mr{Mat}(#1,\, #2 ;\, #3 )}
\newcommand{\transp}{\phantom{}^t}

\begin{document}

\title[Picard groups of Siegel modular 3-folds and $\theta$-lifting]%
{Picard groups of Siegel modular threefolds and theta lifting}

\author[Hongyu He]{Hongyu He$^1$}
\address{Department of Mathematics, Louisiana State University,
Baton Rouge, LA 70803}
\email{\href{mailto:hongyu@math.lsu.edu}{hongyu@math.lsu.edu}}
\urladdr{\href{http://www.math.lsu.edu/~hongyu/}
{http://www.math.lsu.edu/\~{}hongyu}}

\author[J. W. Hoffman]{Jerome William Hoffman}
\address{Department of Mathematics, Louisiana State University,
Baton Rouge, LA 70803}
\email{\href{mailto:hoffman@math.lsu.edu}{hoffman@math.lsu.edu}}
\urladdr{\href{http://www.math.lsu.edu/~hoffman/}%
{http://www.math.lsu.edu/~hoffman/}}

\thanks{{$^1$ Hongyu He supported in part by NSF contract DMS-0700809}}



\begin{abstract}
We show that the Humbert surfaces rationally generate the Picard
groups of Siegel modular threefolds. This involves three ingredients:
(1) R. Weissauer's determination of these Picard groups in terms
of theta lifting from cusp forms of weight $5/2$ on
$\tilde{\mr{SL}}_2 (\R ) $ to automorphic forms on
$\mr{Sp}_4 (\R )$. (2) The theory of special cycles due to Kudla/Millson
and Tong/Wang relating cohomology defined by automorphic forms to
that defined by certain geometric cycles. (3) Results of R. Howe
about the structure of the oscillator representation in this situation.
\end{abstract}

\maketitle

\section{Introduction}
\label{S:intro}
\subsection{}
\label{SS:intro1}
Let $\mfr{H}_g$ be the Siegel half space of genus $g$ and let 
$\Gamma \subset \mr{Sp}_{2g}(\Z)$ be a congruence subgroup.
The quotient $X_{\Gamma} := \Gamma \backslash \mfr{H}_g$ is the set
of complex points of a quasi-projective algebraic variety. These 
varieties are of considerable importance in geometry and arithmetic, but 
they are really only well understood for the case $g=1$, the case of
modular curves. Since the nineteenth century one has known how to
compute their Betti numbers. Also in the nineteenth century it was understood that 
their cohomology was related to modular forms: 
$\mr{H}^0 (X_{\Gamma}, \Omega ^1) \subset \mr{H}^1 (X_{\Gamma},  \C)$
is canonically isomorphic to $S_2 (\Gamma )$, the space of cusp forms 
of weight $2$ for $\Gamma$. More recent is the discovery that certain 
special cycles, modular symbols, provide a good set of homology generators, 
these generators having good transformation properties with respect
to the Hecke algebra. Modular symbols are of great practical 
value in computations with modular forms, providing the key to 
the algorithms of William Stein and others that 
are implemented in software systems. Finally, Eichler and Shimura 
proved that the zeta functions of modular curves are expressible 
in terms of $L$-functions of modular forms.

\subsection{}
\label{SS:intro2}
We know much less even for the case $g=2$, Siegel modular threefolds. 
One cannot compute by any 
practical effective algorithm the Betti numbers of these varieties in general. 
Some cases where the computations have been carried out can be found in \cite{LW1},
\cite{LW2}, \cite{HW1}, \cite{HW2}, \cite{HW3}.
Laumon, \cite{gL}, \cite{gL2}, has proved that the zeta functions of Siegel modular
threefolds are expressible in terms of the $L$-functions of 
automorphic representations, but his 
theorem is limited in an important respect: neither side of this equation 
can be computed exactly except in a very small number of cases 
because the expression of those zeta functions involve multiplicities
which are related to the Betti numbers of those varieties. 
It is true that the cohomology can be described in terms of automorphic forms, 
this being a general fact about quotients of symmetric domains by lattices, 
and for $\mr{H}^2$, one has an explicit description, due to Weissauer.

\subsection{}
\label{SS:intro3}
In this paper we study one piece of this  $\mr{H}^2$, namely the Picard group. Geometrically one can view this either as the group of algebraic line bundles, or as the Chow group of codimension one algebraic cycles modulo rational equivalence. We show that these Picard groups are generated by certain special cycles which classically are known as Humbert surfaces. This is based on three key 
facts:
\begin{enumerate}
\item [1.] Weissauer has shown that 
$\mr{Pic}(X_{\Gamma})\otimes \C = \mr{H}^{1, 1}(X_{\Gamma})$ and that all
the cohomology classes in the complement of the canonical 
polarization can be represented by $(\mfr{g}, K)$-cohomology
classes with values in $\theta (\sigma)$, the space of theta lifts
from holomorphic cusp forms $\sigma$ of weight $5/2$ for the group
$\tilde{\mr{SL}_2}(\R)$ to the group $\mr{Sp}_4(\R)\sim \mr{SO}_0 (3, 2)$.
\item [2.] The theory of special cycles, due largely to two groups:
Kudla-Millson and Tong-Wang, asserts a close connection between 
cohomology classes defined by automorphic forms on locally symmetric
varieties and classes defined by certain geometric cycles on those manifolds. In the case at hand, the connection is between theta lifts
of holomorphic cusp forms of weight $5/2$ and algebraic cycles 
which are combinations of transforms of classical Humbert surfaces 
under the Hecke algebra. Here it is crucial that we
are in a stable range: $1 < (3 +2)/4$ (see theorem 
\ref{T:KM2}).    
\item [3.] The main issue is then to see that the {\it general}
theta lifts $\theta (\sigma)$ occurring in Weissauer's theorem 
are in the span of the {\it special} theta lifts 
$\theta _{\mr{special}}(\sigma)$ occurring in the theory of KM and TW. 
This is a problem about the oscillator representation:
the theta kernel $\theta _{\mr{special}}$ is characterized
by representation-theoretic properties. We apply general structure
theorems about the oscillator representation and dual reductive pairs 
due to Howe to conclude our result. 
\end{enumerate}

\subsection{}
\label{SS:intro4}
For the convenience of the reader, sections \ref{S:special} through 
\ref{S:cohrep} collect some background 
utilized in sections \ref{S:siegel} through 
\ref{S:main} where the proofs of the main results are given.
The reader is warned that the notation sometimes changes (for instance
the letter $V$ is a local system in section \ref{S:cohauto}; a real vector space
of dimension $p+q$ in sections \ref{S:special}, 
\ref{S:theta}; a rational vector space of dimension $4$ in section
\ref{S:iso};
a real vector space of dimension $2n$, in section \ref{S:sym}). Eventually
these are specialized to $(p, q) = (3, 2)$, $n=1$.  This is done in part
to be consistent with the notation in the references. Also, the papers of 
Weissauer use the adelic point of view, and the group of symplectic 
similitudes, whereas the papers of Kudla-Millson and Tong-Wang
use the classical viewpoint of automorphic forms as functions on real Lie
groups invariant under a lattice (one exception: \cite{K}, but that paper
deals only with the anisotropic case, i.e., compact quotients). This 
necessitates a discussion of the connection between them in section \ref{S:adelic}. 

\section{Siegel modular threefolds}
\label{S:siegel}
\subsection {}
\label{SS:siegel1}

Let $G = \Sym{4}{\R }$ be the group of real symplectic matrices
of size four. This acts on the Siegel space
\[
\Sg{2} \ =\ \{ \tau \in {\mathbf  M}_2 (\C ) : \phantom{}^t \tau =
\tau ,\
           \operatorname {Im} (\tau ) \text{ is positive definite }  \}
\]
via
\[
g . \tau \ =\ ( a \tau  + b) ( c \tau + d ) ^{-1}, \quad g
\ =\
                        \begin{pmatrix}
                                 a & b \\
                                 c & d
                        \end{pmatrix}.
\]
The group $G/\pm 1$ is the group of holomorphic automorphisms
of $\Sg{2}$, and it acts transitively. The stabilizer
$K$ of the point $i \mbf{1}_2 = \sqrt{-1}\mbf{1}_2$ is
$\{ k =  \begin{pmatrix}
                                 a & b \\
                                 -b & a
                        \end{pmatrix} \}$, which is isomorphic
to the unitary group $U(2)$ via $k \mapsto a + bi$. Thus
$\Sg{2}$ is the symmetric space $G/K$ attached to $K$.
Reference: \cite{HKW}. 
\subsection{}
\label{SS:siegel2}
Let $\Gamma \subset \Sym{4}{\Q}$ be a subgroup commensurable with
$\Sym{4}{\Z}$. Then $\Gamma$ is an arithmetic group. According to a
theorem of Baily-Borel, $X_{\Gamma} = \Gamma \backslash \Sg{2}$
is the analytic space attached to the set of $\C$-points of a
quasi-projective algebraic variety defined over $\C$. The
principal congruence subgroup of level $N$, for an integer
$N \ge 1$, is defined as
\[
\Gamma (N) = \{ \gamma \in \Sym{4}{\Z} \mid
\gamma \equiv \mbf{1}_4 \mr{\ mod\ } N \}.
\]
Every subgroup $\Gamma \subset \Sym{4}{\Z}$ of finite index is a congruence
subgroup in the sense that $\Gamma \supset \Gamma (N)$ for some $N$.
The spaces $X_{\Gamma}$ admit several compactifications: the Borel-Serre
compactification (which is a manifold with corners); the
Satake compactification, which is a projective variety, but usually singular; the toroidal compactifications, which are often smooth and
projective. For a modern discussion of compactifications
of quotients of bounded symmetric domains, see \cite{BJ}.
The general theory of Siegel modular
varieties and their compactifications can be found in
\cite{FC}. 

\subsection{}
\label{SS:siegel3}
The algebraic variety $X$ whose analytic space is
$X^{an} = X(\C ) = \Gamma\backslash\Sg{2}$, is defined
a priori over $\C$, but in fact has a model defined over
a number field (a finite extension of $\Q$).
This
can be seen from two points of view: 
\begin{enumerate}
\item [1.] $X$ is a moduli space for systems $(A, \Phi)$ where
$A$ is an abelian variety of dimension 2, 
and $\Phi$ consists of additional
structures on $A$, typically polarizations, endomorphisms,
rigidifications of points of some order $N$. The theory of moduli
spaces then provides a structure of a scheme (or more generally,
stack) over a number field. 
\item[2.]  $X$ is a Shimura variety. Shimura varieties arise as
quotients of hermitian symmetric spaces by arithmetic
groups with certain additional properties. It is known that
these have canonical models defined over algebraic number fields.
For an introduction to this see \cite{Mi}.
\end{enumerate}

\section{Cohomology and automorphic forms}
\label{S:cohauto}
\subsection{}
\label{SS:cohauto1}

Let $X_{\Gamma} = \Gamma\backslash\underline{G}(\R)/K$,
where $\underline{G}$ is a semisimple 
(more generally: reductive) algebraic group
defined over $\Q$, $K \subset \underline {G}(\R)$ is a
maximal compact subgroup, and
$\Gamma \subset \underline{G}(\Q)$ is a congruence subgroup.
If $\underline{V}$ is a local system of complex vector
spaces coming from a rational finite dimensional representation
$V = V_{\mu}$
of $\underline{G}$, then it is known that the cohomology
$\coh{n}{X_{\Gamma}}{\underline{V}}$ is computable in terms
of automorphic forms. There is a canonical isomorphism
\[
\coh{n}{X_{\Gamma}}{\underline{V}} = \coh{n}{\Gamma}{V_{\mu}} =
\cohb{n}{\mfr{g},\, K}
{ \mr{C}^{\infty}(\Gamma \backslash G)\otimes V_{\mu}}
\]
where the right-hand side is relative Lie algebra cohomology
(see \cite{BW}). The major result, due Jens Franke, 
\cite{F}, built
on earlier works by Borel, Casselman,
Garland, Wallach, is that these spaces are all isomorphic
to $\cohb{n}{\mfr{g},\, K}
{ \mc{A}(\Gamma \backslash G)\otimes V_{\mu}}$, where
$ \mc{A}(\Gamma \backslash G)\subset \mr{C}^{\infty}(\Gamma \backslash G) $
is the subspace of automorphic forms (see \cite{BJa} for the definition
of this space; see \cite{LS} and \cite{Wa} for a description and
survey of this result, and \cite{FS} and \cite{Na} for refinements
and generalizations).

\subsection{}
\label{SS:cohauto2}
As before, $G = \underline{G}(\R)$ for a semisimple algebraic
group $\underline{G}$, and let $\Gamma \subset G$ be a lattice
(a discrete, finite covolume subgroup). Let
$L^2 _{\mr{disc}}(\Gamma\backslash G)$ be the discrete part
of the $G$-module  $L^2 (\Gamma\backslash G)$, so that we have
\[
L^2 _{\mr{disc}}(\Gamma\backslash G) =
\bigoplus_{\pi \in \hat{G}}\mr{m}(\pi,  \Gamma ) H_{\pi}
\]
where the sum is over the irreducible unitary representations,
and the multiplicities $\mr{m}(\pi,  \Gamma )$ are finite. If
$\Gamma$ is cocompact or $G$ has a compact Cartan subgroup, there
is an isomorphism
\[
\mr{IH} ^n (X_{\Gamma },\, \underline{V}) =
\mr{H}^n_{(2)}(X_{\Gamma },\, \underline{V})=
\bigoplus_{\pi \in \hat{G}}\mr{m}(\pi,  \Gamma )
\cohb{n}{\mfr{g}, K}{H_{\pi}\otimes V_{\mu}}
\]
where the left-hand term is intersection cohomology, the middle
term is $L^2$-cohomology; the first equality records the solution
to the Zucker conjecture. The determination of which irreducible
unitary $\pi$ have nonzero $(\mfr{g}, K)$-cohomology is due to
Vogan and Zuckerman, \cite{VZ}. These are the representations
denoted by $A_{\mfr{q}} (\lambda )$ for $\theta$-stable parabolic subalgebras
$\mfr{q} \subset \mfr{g} = \mfr{g}_0 \otimes \C$,
$\mfr{g}_0 = \mr{Lie}(G)$, and certain
linear forms $\lambda$ on a Levi factor $\mfr{l}\subset \mfr{q}$.

\subsection{}
\label{SS:cohauto3}

For  $G = \mr{GSp}_4(\R )$, which has a compact Cartan subgroup, 
the representations with nonzero
$(\mfr{g}, K)$-cohomology have been determined. The list,
without proofs, can be found in \cite{rT}. Because of the
isogeny $\Sym{4}{\R}\sim \mr{SO}_0 (3, 2)$ (see section \ref{S:iso})
these representations can be described in orthogonal language
and have been listed in \cite{HL}, \cite{LS}. The important case
for us are those that contribute to the to the Hodge (1, 1) part.
These are: the trivial representation, and two others
$\pi ^{\pm}$. These last two differ by twist:
$\pi ^{-} = \pi ^{+}\otimes (\mr{sgn}\circ \nu)$ where
$\nu :  \mr{GSp}_{4} \to \mbf{G}_m$ is the canonical
character with kernel $\mr{Sp}_4$. For more details, 
see section \ref{S:cohrep}.

\section{Weissauer's Theorems}
\label{S:weiss}
For Siegel modular threefolds, and
$\underline{V}_{\mu} = \C$, Weissauer
has completely analyzed $\coh{2}{X_{\Gamma}}{\C}$ in terms
of automorphic forms. See his papers \cite{rW2}, \cite{rW4}.
First, the cohomology is all square-integrable, in fact:
\begin{thm}
\label{T:weiss1}
\[
\quad
\mr{H}^2(X_{\Gamma },\, \C)= \mr{H}^2_{(2)}(X_{\Gamma },\, \C)
= \mr{IH}^2 (X_{\Gamma },\, \C).
\]
\end{thm}
This theorem shows in particular that $\mr{H}^2(X_{\Gamma },\, \C)$
has a pure Hodge structure of weight 2.
Consider the Hodge decomposition
\[
\mr{H}^2(X_{\Gamma },\, \C)=
\mr{H}^{2, 0}(X_{\Gamma })\oplus
\mr{H}^{1, 1}(X_{\Gamma })\oplus \mr{H}^{0, 2}(X_{\Gamma }),
\quad
\mr{H}^{0, 2} = \overline{\mr{H}^{2, 0}}.
\]
Via the isomorphism $\mr{H}^2(X_{\Gamma },\, \C) = 
\cohb{2}{\mfr{g},\, K}{ \mc{A}(\Gamma \backslash G)}$
recalled in section \ref{SS:cohauto1}, we can describe the cohomology
in degree 2 as certain kinds of closed differential forms on
$\mfr{H}_2$ with automorphic form coefficients. The automorphic
forms that contribute are square-integrable. More specifically 
one has: 
\begin{thm}
\label{T:weiss2}
The automorphic forms contributing to  $\mr{H}^{2, 0}(X_{\Gamma })$
are given by theta lifting from
dual reductive pairs
$(\mr{GO} (b), \mr{Sp}_4 (\R))$ where $b$ are
two dimensional positive-definite quadratic forms
defined over $\Q$. This allows for an explicit computation
of $\dim \mr{H}^{2, 0}(X_{\Gamma })$. See \cite{rW3}. 
\end{thm}

This has been generalized in part by
Jian-Shu Li to $\mr{H}^{g, 0}(X_{\Gamma })$
for quotients of $\Sg{g}$, see \cite{Li1}.

\begin{thm}
\label{T:weiss3}
\begin{itemize}
\item [1.] Siegel modular threefolds have maximal Picard number:
\[
\mr{H}^{1, 1}(X_{\Gamma })  = \mr{Pic} (X_{\Gamma })\otimes \C .
\]

\item[2.] There is a canonical decomposition
$\mr{Pic} (X_{\Gamma })\otimes \C = \C \cdot [\mc{L}]\oplus
\mr{Pic} (X_{\Gamma })_0$ where $[\mc{L}]$ is the
Lefschetz class. Then:
\begin{itemize}
 \item[2.1]  $[\mc{L}]$  corresponds to the trivial
automorphic representation of $\mr{Sp}_4$.
\item[2.2] The automorphic forms in  $\mr{Pic} (X_{\Gamma })_0$
are given by theta lifting from
the dual reductive pair
$(\tilde{\mr{SL}}(2, \, \R ), \mr{SO}_0 (3, 2)\sim \mr{Sp}_4 (\R))$.
More precisely, they are all given by lifting of weight $5/2$
holomorphic cusp forms on $\tilde{\mr{SL}}(2, \, \R )$
to the unique automorphic representation
of $\mr{SO}_0 (3, 2)$
that contributes to  $\mr{Pic} (X_{\Gamma })_0$ by
Vogan-Zuckerman theory.
\end{itemize}
\end{itemize}
\end{thm}

\section{Adelic formulation}
\label{S:adelic}
Let $\mbb{A} = \R \times \mbb{A}_f$ be the ring of adeles of
the rational field $\Q$; $\mbb{A}_f = \Q \otimes \prod _p \Z _p $
is the ring of finite adeles.
\subsection{}
\label{SS:adelic1}
Let $G = \mr{GSp}_4$ be the algebraic group over $\Q$ of
symplectic similitudes, i.e., of $4\times 4$ matrices
$g$ such that
\[
\transp{g} \Psi g = \nu (g) \Psi, \quad
\Psi = \begin{pmatrix}
     0 & 1 _2\\
      -1 _2 & 0
    \end{pmatrix}.
\]
There is an exact sequence
\[
\begin{CD}
0 @>>> \mr{Sp}_4@>>> \mr{GSp}_4@>\nu >> \mbb{G}_m @>>> 0.
\end{CD}
\]
Let
\[
h : \mbb{S} := \mr {Res}_{\R}^{\C}( \mbb{G}_m) \to \mr{GSp}_4
\]
be the morphism defined over $\R$ with the property that
$x + i y \in \C ^{\times} =  \mbb{S}(\R)$ maps to
\[
\begin{pmatrix}
     x1 _2 & y1 _2\\
      -y 1_2 &  x1 _2
    \end{pmatrix}.
\]
Let $K_{\infty} \subset \mr{GSp}_4 (\R )$ be the stabilizer of $h$.
Then $K_{\infty} = Z_{\R}\cdot K'_{\infty} $, where
$Z_{\R}\subset \mr{GSp}_4 (\R )$ is the center and
$K'_{\infty}\subset \mr{Sp}_4 (\R )$ is a maximal compact
subgroup. For any open subgroup of finite index
$L \subset  \mr{GSp}_4 (\mbb{A}_f)$ we define
\[
 M_L (\C) =  M_L (\mr{GSp}_4 (\Q), h)_{\mr{an}}=
\mr{GSp}_4 (\Q)\backslash \mr{GSp}_4 (\mbb{A})/K_{\infty}L.
\]
This is the set of complex points of a quasiprojective algebraic
variety $ M_L$ defined over a number field. This is a
disjoint union of spaces of the type $X_{\Gamma}$ discussed above,
for various arithmetic subgroups $\Gamma \subset  \mr{Sp}_4 (\Q)$.
For instance, if we take, for an integer $N\ge 1$,
\[
L_N = \left \{ k \in \prod _p G(\Z _p) \mid
k \cong \ 1_4\  \mr{mod} \ N
       \right \}
\]
then $M_{L_N}(\C) := M_N (\C)$ is a disjoint union
of $\phi (N)$ copies of $\Gamma (N)\backslash \Sg{2}$. The variety
$M_N$ is defined over $\Q$, and each connected component is
defined over $\Q (\zeta _N )$, $\zeta _N = \mr{exp}(2 \pi i /N)$.

\subsection{}
\label{SS:adelic2}
Recall that $G = \mr{GSp}_4$. We define
\[
\mr{H}^i(\mr{Sh}(G), \C)  :=  \underset{L}{\varinjlim}\,
\mr{H}^i( M_L(\C), \C)
\]
which is in a canonical way an admissible
$\pi _0 (G(\R))\times G (\mbb{A}_f)$-module. For any compact open
subgroup $L$ we have
\[
\mr{H}^i(\mr{Sh}(G), \C) ^{L} = \mr{H}^i( M_L(\C), \C).
\]
This is a module for the Hecke algebra
$\mathcal{H}_{L} = C_c (G(\mbb{A}_f)//L)$ of $\C$-valued
compactly supported $L$-biinvariant functions on
$G(\mbb{A}_f)$, which is an algebra for the convolution
product, once a Haar measure is fixed on $G(\mbb{A}_f)$.
The major result is that there is a canonical isomorphism
\[
\mr{H}^i(\mr{Sh}(G), \C) = \mr{H}^i (\mfr{g}, K_{\infty}; \mc{A}(G)),
\]
where $\mfr{g} = \mr{Lie} (G(\R)) = \mfr{gsp}_4 (\R)$,
$K_{\infty}$ is defined in section \ref{SS:adelic1},
$\mc{A}(G)$ is the space of automorphic forms on
$G(\mbb{A})$, and the right-hand side is relative
Lie algebra cohomology. This is an isomorphism
of $G(\mbb{A}_{f})$-modules, for the canonical
structures on both sides. In this case, the above
isomorphism can be refined to an isomorphism of
Hodge $(p, q)$-components.

\subsection{}
\label{SS:adelic3}
Weissauer's theorems are the following:

\subsubsection{}
\label{SSS:adelic3.1} $\mr{H}^2(\mr{Sh}(G), \C)  = \mr{H}^2 _{(2)}(G, \C)$.
Therefore we have, for each Hodge index $(p, q)$ with $p+q=2$,
\[
\mr{H}^{p, q}( M_L(\C)) =
\bigoplus _{\pi _{\infty} \in \mr{Coh}^{p, q}}
m(\pi )\mr{H}^{p, q} (\mfr{g}, K_{\infty}; \pi _{\infty})\otimes \pi _f^{L}
\]
where the sum ranges over all the irreducible automorphic representations
$\pi = \pi_{\infty}\otimes \pi _f$ which occur in the discrete spectrum
\[
L^2 _{d}(G(\Q)Z(\R)^{\circ}\backslash G(\mbb{A}), dg)
\]
where $Z(\R)^{\circ}\subset G(\R)$ is the connected component of the center.
The set $\mr{Coh}^{p, q}$ is the finite
set of unitary representations of $G(\R)$ with trivial central character and
with nonzero $(\mfr{g}, K_{\infty})$-cohomology in 
dimension $(p, q)$.

\subsubsection{}
\label{SSS:adelic3.2}
There is only one element of $\mr{Coh}^{2, 0}$ (resp. $\mr{Coh}^{0, 2}$),
call it $\pi $. Then
$\mr{H}^{p, q} (\mfr{g}, K_{\infty}; \pi )$ is one-dimensional
for $(p, q) = (2, 0)$ (resp. $(0, 2))$. Every automorphic representation
contributing to $\mr{H}^{2, 0}(M_L(\C), \C)$ is in
the image of the theta lifting from the orthogonal similitude group
$\mr{GO}(b)$ as $b$ ranges over the positive-definite binary quadratic
forms over $\Q$.

\subsubsection{}
\label{SSS:adelic3.3}
$\mr{Coh}^{1, 1} = \{1, \pi^{\pm} \}$, where $1$ is the trivial one-dimensional
representation, and $\pi ^{-} = \pi ^{+}\otimes \mr{sgn}$, where
$\mr{sgn }: \mr{GSp}_4 (\R ) \to \{\pm 1\}$ is the sign character. One has
\[
\mr{Pic}( M_L (\C))\otimes \C =
\mr{H}^{1, 1}(M_L(\C)))
\]
and we can canonically decompose this as
\[
\mr{Pic}( M_L(\C))\otimes \C =
\C .[\mc{L}]\oplus \mr{Pic}( M_L(\C))_0 \otimes \C
\]
where $\mc{L}$ is the canonical polarization (``Lefschetz class'').
This term corresponds to the automorphic representation $1$.
Weissauer showed that the classes in 
$\mr{Pic}(M_L(\C))_0 \otimes \C  = \mr{H}^{1, 1} (M_L(\C))_0$ in the complement of the Lefschetz
class are generated by the images of 
$\mr{H}^{1, 1}(\mfr{g}, K; 
\theta (\sigma, \psi)\otimes (\chi\circ \nu))$. Here, $\psi : \mbb{A}/\Q \to \C$ is a nontrivial
additive character, $\sigma$ is an irreducible (anti)holomorphic
cuspidal automorphic representation of $\tilde{\mr{SL}}_2(\mbb{A})$ 
of weight $5/2$, $\theta (\sigma, \psi)$ is the theta lifting
with respect to the Weil representation $\omega _{\psi}$
to an automorphic representation to $\mr{PGSp}_4 (\mbb{A})$
viewed as a representation of $\mr{GSp}_4 (\mbb{A})$, 
$\chi : \mbb{A}^{*}/\Q ^{*}\R _{> 0}^{*} \to \C^{*}$ is an
idele class (Dirichlet) character, and 
$\nu: \mr{GSp}_4 \to \mbf{G} _m$ is the canonical character
with kernel $\mr{Sp}_4$.

\subsection{}
\label{SS:adelic4}
We get the Picard group by varying 
all the data in the above. 
First note that we can fix one choice of nontrivial additive character $\psi$. 
The reason is that, every other nontrivial additive character
is of the form $\psi _t$ for a $t \in \Q ^*$, where
$\psi _t (x) = \psi (tx)$. It is known that 
$\theta (\sigma, \psi _t) =  \theta (\sigma _t, \psi)$ (~\cite{rH1}, 1.8), where
$\sigma _t $ is the automorphic representation 
\[
g \mapsto \sigma \left ( 
\begin{pmatrix}t &0\\0 &1 \end{pmatrix}
g  \begin{pmatrix}t^{-1} &0\\0 &1 \end{pmatrix}
\right ). 
\]
So from now on, we drop explicit reference to $\psi$.

\subsection{}
\label{SS:adelic5}
For each integer $N\ge 1$ let $M_N  = M _{K_{N}}$, 
which is a scheme over $\Q$. $M_N \otimes \overline{\Q}$ 
has $\phi (N)$ (Euler phi) connected components, defined and
all isomorphic over $\Q (\zeta _{N})$, where $\zeta _N$ is a 
primitive $N^{\mr{th}}$ root of unity. These are permuted simply
transitively by $\mr{Gal}(\Q (\zeta _{N})/\Q)$. We denote 
any one of these components by $M_N ^0$. We have 
$M_N ^0 (\C) = \Gamma (N)\backslash \mfr{H}_2$. 
The group $\mr{Gal}(\overline{\Q} /\Q)$ acts on  
$\mr{Pic}(M _N)\otimes \Q$, fixing the Lefschetz class.
Weissauer proved \cite[Thm. 2, p. 184]{rW4} that the 
action of $\mr{Gal}(\overline{\Q} /\Q)$ on $\mr{Pic}(M _N)\otimes \Q$
factors over the abelian quotient $\mr{Gal}(\Q ^{\mr{ab}} /\Q)$. 
Therefore we have a decomposition
\[
\mr{Pic}(M _N)_0\otimes \C = \bigoplus _{\chi}
\mr{Pic}^{\chi}(M _N)_0
\]
of isotypical spaces for the characters $\chi$ of  
$\mr{Gal}(\Q ^{\mr{ab}} /\Q)$. By classfield theory these
can be identified with idele class characters
$\chi : \mbb{A}^{*}/\Q ^{*}\R _{> 0}^{*} \to \C^{*}$. 
The space $\mr{Pic}( M_N(\C))_0$ is the kernel of the canonical 
trace map
\[
\mr{Pic}( M_N(\C))\otimes \Q \to \mr{Pic}( M_1(\C))\otimes \Q.
\]
We can similarly define $\mr{Pic}( M_N ^0(\C))_0$.
Evidently, for the inclusion of any connected 
component $M_N ^0 \to M_N$, the restriction
$\mr{Pic} (M _N)_0 \otimes \Q\to 
\mr{Pic}(M _N ^0)_0\otimes \Q$ is surjective. By 
choosing these inclusions compatibly we can define a map
\[
\mr{Pic} (M) := \underset{N}\varinjlim\
\mr{Pic} (M _N)_0 \otimes \Q \to 
\underset{N}\varinjlim \ \mr{Pic} (M _N ^0)_0 := \mr{Pic} (M^0). 
\]

\begin{lem}
\label{L:adelic1} For the identity character $1$, the map
$\mr{Pic}^{1} (M)_0 \to 
\mr{Pic} (M ^0)_0\otimes \C$ is surjective.
\end{lem}
\begin{proof}
Let $G_n$ be the kernel of the map  
$\mr{Gal}(\Q ^{\mr{ab}} /\Q) \to \mr{Gal}(\Q (\zeta _{n})/\Q)$. 
Let $\mc{M}\in\mr{Pic} (M _N ^0)_0 $. Then there is an 
$N'\ge N$ with the property that 
$\mr{Gal}(\overline{\Q}  /\Q (\zeta _{N'}))$ acts
trivially on $i_{N, \ast}\mc{M}$, where
$i_N : M_N ^0 \to M_N$ is the inclusion
(extension by $0$ on all the other components).   
Let $f : M_{N'}^0 \to M_N ^0$ be the canonical projection, 
and extend $i_N$ to a map $i_{N'}:  M_{N'} ^0 \to M_{N'}$
which commutes with the projection $f : M_{N'} \to M_N $. 
The line bundle $ \mc{M}' = i_{N', \ast} f^* \mc{M}$ 
is fixed by $G_{N'}$, and hence for any 
$g \in \mr{Gal}(\Q (\zeta _{N'})/\Q)$, $g^* \mc{M}'$ is 
well-defined. Then we define 
$\mc{N}\in \mr{Pic} ^{1}(M_{N'})_0$ by 
\[
\mc{N} = \sum _{g \in \mr{Gal}(\Q (\zeta _{N'})/\Q)} g^*\mc{M}'.
\]
Moreover, since $\mr{Gal}(\Q (\zeta _{N'})/\Q)$
acts simply transitively on the components of 
$M_{N'}$, it follows that $i^{\ast}_{N'}\mc{N} = f^* \mc{M}$.
This shows that after extension to $N'$ the class
$\mc{M}\in  \mr{Pic} (M ^0 _N)_0 \subset \mr{Pic} (M^0)_0$ 
is in the image of $\mr{Pic}^{1} (M_{N'})_0 \subset \mr{Pic}^{1} (M)_0$.
\end{proof}

\begin{lem}
Any element of $\mr{Pic} (M^0)_0\otimes \C$ is in the image 
(in the sense of section \ref{SS:adelic4}) of the Saito-
Kurokawa lifts $\theta (\sigma)$ of holomorphic 
cusp forms $\sigma$ of weight $5/2$.
\end{lem}
\begin{proof}
Weissauer showed more precisely that the elements
of $\mr{Pic} ^{\chi}(M )_0$ are in the image of
$\theta (\sigma, \psi)\otimes (\chi\circ \lambda)$.
But lemma \ref{L:adelic1} shows that every element
of $\mr{Pic} (M^0)_0\otimes \C$ is in the image of 
$\mr{Pic} ^{1}(M )_0$ and these are in the image of the 
Saito-Kurokawa lift. 
\end{proof}

\section{Structure of the oscillator representation}
\label{S:fock}
This section follows \cite{rH2}, \cite{KM4}. We let $V$ be the $\R$-vector
space with a quadratic form $(\ , \ )$ of signature $(p, q) = (3, 2)$;
$n= 5 = p+q$.
Let $W$ be the $\R$-vector space of dimension $m = 2$ with an alternating
nondegenerate bilinear form $\langle \ ,\ \rangle$. We describe a model for
the infinitesimal oscillator representation $(\omega, \mfr{sp} (V\otimes W)$.
The maximal compact subgroup of
$\mr {Sp}_{10} = \mr{Sp} (V\otimes W)\sim \mr{Sp} (10, \R)$ is isomorphic to
the unitary group $U_5$. We let $\tilde{\mr {Sp}}_{10} = \tilde{\mr{Sp}} (10, \R)$ be
the metaplectic cover.
and in general putting a tilde over an object in  $\mr{Sp} (10, \R)$
denotes its inverse image in the metaplectic cover.

\subsection{}
\label{SS:fock1}
The space of $\tilde {U}_{5}$-finite vectors in the Fock realization of $\omega$
is isomorphic to the space of polynomials $\mathcal{P}(\C ^5)$ in five
variables $z_i$, $i=1,...,5$. The action is given by
\begin{align*}
\omega ( \mfr{sp}_{10}\otimes \C) &=
\mfr{sp} ^{(1, 1)}\oplus \mfr{sp}^{(2, 0)}\oplus\mfr{sp}^{(0, 2)}\\
\mfr{sp} ^{(1, 1)} &= \mr{span\ of\ } \left \{
                \left (
    z_i \frac{\partial }{\partial z_j} + \frac{1}{2} \delta_i^j
                            \right ) \right \}\\
\mfr{sp} ^{(2, 0)} &= \mr{span\ of\ } \left \{
    z_i  z_j \right \}\\
\mfr{sp} ^{(0, 2)} &= \mr{span\ of\ } \left \{
   \frac{\partial ^2}{ \partial z_i \partial z_j }\right \}.
\end{align*}
In the Cartan decomposition $ \mfr{sp}_{10} = \mfr{u}_5 \oplus \mfr{q}$
we have
\[
\omega (\mfr{u}_5 \otimes \C) = \mfr{sp} ^{(1, 1)}, \quad
\omega (\mfr{q} \otimes \C) = \mfr{sp} ^{(2, 0)}\oplus \mfr{sp} ^{(0, 2)}.
\]

\subsection{}
\label{SS:fock2}
We are interested in the reductive dual pair
\[
(G, G')= (\tilde{O}(V) = \tilde {O}_{3, 2},\
\tilde{Sp}(W) = \tilde {SL}_{2}(\R ) )
\]
inside $\tilde{Sp} (V\otimes W) = \tilde{Sp}_{10}$, and especially the
the structure of $\mc{P}(\C ^5)$ as a
$(\mfr{g}, \tilde{K})\times (\mfr{g}', \tilde{K}')$-module,
where $\mfr{g} = \mr{Lie}(G) = \mfr{o}(V) = \mfr{o}_{3, 2}$,
$\mfr{g}' = \mr{Lie}(G') = \mfr{sp}(W) = \mfr{sl}_{2}(\R)$,
$K = \mr{O}(3) \times \mr{O}(2)$ is the maximal compact
subgroup of $G$, $K'= \mr{SO}(2)$ is the maximal compact
subgroup of $G'$. Following the convention in \cite[p. 154]{KM4}
we number the variables $z_i$ as $z_{\alpha}$, $\alpha = 1, 2, 3$,
and $z_{\mu}$, $\mu = 4, 5$; generally indices $\alpha, \beta, ...$
run from 1 to 3 and indices $\mu, \nu, ...$ run from 4 to 5. In this
numbering the group $ \mr{O}(3) \times \mr{O}(2)$ acts
so that $\mr{O}(3)$ rotates the variables $z_{\alpha}$ and
 $\mr{O}(2)$ rotates the variables $z_{\mu}$.

\subsection{}
\label{SS:fock3}
Let
\[
\mc{P} = \mc{P}(\C ^5)=
\bigoplus _{\sigma \in \mc{R}(\tilde{K}, \omega)} \mathscr{I}_{\sigma}
\]
be the decomposition into $\tilde{K}$-isotypical components; the
notation $\mc{R}(\tilde{K}, \omega)$ refers to the isomorphism
classes of representations of $\tilde{K}$ that occur in the oscillator
representation. We recall the definition of harmonics. The Lie algebra
$\mfr{k} = \mfr{o}_{3}\times \mfr{o}_2$ of the maximal compact
 subgroup of $G = O(V)$ is a member of a dual reductive pair
$(\mfr{k}, \mfr{l}')$. In this case,
$\mfr{l}' = \mfr{sl}_2 (\R) \times  \mfr{sl}_2 (\R)$. We can decompose
\[
\mfr{l}' = \mfr{l} ^{'(2, 0)}\oplus \mfr{l}^{'(1, 1)}\oplus \mfr{l}^{'(0, 2)},
\quad \mr{where\ \ \ } \mfr{l} ^{'(i, j)}=  \mfr{l}'\cap \mfr{sp}^{(i, j)}.
\]
Then the harmonics are defined by
\begin{align*}
\mathscr{H}(K)&=\mathscr{H}(\tilde{K})=
\left \{
P \in \mc{P} : l(P) = 0  \mr{\ for\ all\ } l \in \mfr{l}^{'(0, 2)}
\right \}
\\
\mathscr{H}(K)_{\sigma} &= \mathscr{H}(K)\cap \mathscr{I}_{\sigma}
\end{align*}
The crucial point for us the Howe's result \cite[p.542]{rH2}:
\begin{thm}
\label{T:howe}
For each $\sigma \in \mc{R}(\tilde{K}, \omega)$:
\begin{enumerate}
\item[1.] The space $\mathscr{H}(K)_{\sigma}$ consists precisely
of the polynomials of lowest degree in $\mathscr{I}_{\sigma}$;
these are homogeneous all of the same degree, $\deg (\sigma)$.
\item[2.]
\[
\mathscr{I}_{\sigma} = \mathscr{U} (\mfr{g}')\cdot \mathscr{H}(K)_{\sigma}
= \mathscr{U} \left (\mfr{l}^{'(2, 0)}\right )\cdot \mathscr{H}(K)_{\sigma}
\]
where $\mathscr{U}$ denotes the universal enveloping algebra of the
respective Lie algebra.
\end{enumerate}
\end{thm}

\subsection{}
\label{SS:fock4} 
Let $D$ be the Hermitian symmetric domain attached to
the Lie group $\mr{SO}_0 (3, 2)$. This is isomorphic to the Siegel
half space $\Sg{2}$ via the isogeny $\mr{Sp}_4 (\R)) \to \mr{SO}_0 (3, 2)$.
The tangent bundle to $D$ is the homogeneous vector bundle associated
to the action of the maximal compact $K_0 = \mr{SO}(3)\times \mr{SO}(2)$
on $\mfr{p} := \mfr{g}/\mfr{k}$ via the adjoint representation. The complex
structure is given by the action of the subgroup
$\mr{SO}(2)\subset \mr{SO}(3)\times \mr{SO}(2)$, and we have a decomposition
into Hodge types $\mfr{p}^*_{\C}  = \mfr{p}^{(1, 0)}\oplus \mfr{p}^{(0, 1)}$,
where
\[
k(\theta) = \begin{pmatrix}
            \cos (\theta) & \sin(\theta)\\
             -\sin(\theta)& \cos(\theta)
            \end{pmatrix} \in \mr{SO}(2)
\]
acts as $\exp(i \theta)$ (resp. $\exp( -i \theta)$ )
on $\mfr{p}^{(1, 0)}$ (resp. $\mfr{p}^{(0, 1)}$).
We are interested in the bundle of $(1, 1)$-forms on $D$
which is a subrepresentation
\[
\wedge ^{1, 1}\mfr{p}^*\subset \wedge ^{2}\mfr{p}^*_{\C};
\]
in fact, as $\mr{SO}(3)\times \mr{SO}(2)$-module, the $\mr{SO}(2)$-factor
acts trivially, and the $\mr{SO}(3)$-module splits as $\mbf{1}\oplus\mbf{3}\oplus \mbf{5}$
where, for each odd integer $i$, $\mbf{i}$ is the unique
corresponding irreducible
representation of $\mr{SO}(3)$. Thus, as $\mr{SO}(3)\times \mr{SO}(2)$-module,
\[
\wedge ^{1, 1}\mfr{p}^* \cong
\mbf{1}\otimes \mbf{1}\oplus\mbf{3}\otimes \mbf{1}\oplus\mbf{5}\otimes \mbf{1}.
\]
As these are self-dual, we have the same decomposition for
$\wedge ^{1, 1}\mfr{p}$.
The theta-lifting kernels relevant to us will define classes in
$\mr{H}^{1, 1}(\mfr{g}, K; \mathcal{P}(\C ^5))$, which is a subquotient
of
\[
\mr{Hom}_{K} (\wedge ^{1, 1}\mfr{p}, \mathcal{P}(\C ^5)) =
\mr{Hom}_{K} (\wedge ^{1, 1}\mfr{p}, \mathscr{I}_{\mbf{1}\otimes \mbf{1}} )
\oplus
\mr{Hom}_{K} (\wedge ^{1, 1}\mfr{p}, \mathscr{I}_{\mbf{3}\otimes \mbf{1}} )
\oplus
\mr{Hom}_{K} (\wedge ^{1, 1}\mfr{p}, \mathscr{I} _{\mbf{5}\otimes \mbf{1}})
\]
In fact, only the first and last summand above will contribute to the Picard
group, as we will see. 
Since these isotypical spaces are generated by their harmonics, we need to
analyze those.
\begin{prop}
\label{P:harmonics}
\begin{enumerate}
\item[1.] $\mathscr{H}(K)_{\mbf{1}\otimes \mbf{1}}$ is the one dimensional
$\C$-vector space spanned by $1\in \mathcal{P}(\C^5)$.
\item [2.]  $\mathscr{H}(K)_{\mbf{3}\otimes \mbf{1}}$ is the three dimensional
$\C$-vector space spanned by the $z_{\alpha}\in \mathcal{P}(\C^5)$.
\item[3.] $\mathscr{H}(K)_{\mbf{5}\otimes \mbf{1}}$ is the five dimensional
$\C$-vector space consisting of quadratic forms
\[
\sum _{\alpha, \beta  = 1}^3 c_{\alpha \beta}z_{\alpha }z_{\beta},
\quad c_{\alpha \beta} = c_{\beta\alpha} \in \C,
\quad
\mr{with\ \ } \sum _{\alpha = 1}^3 c_{\alpha \alpha} = 0.
\]
\end{enumerate}
\end{prop}
\begin{proof}
Define
\begin{alignat*}{3}
X_{\alpha}&= -\frac{1}{2} \sum_{\alpha = 1}^{3} z_{\alpha}^2,  &
\quad
Y_{\alpha}&= \frac{1}{2} \sum_{\alpha = 1}^{3}
\frac{\partial ^2}{\partial z_{\alpha}^2}, &
\quad
H_{\alpha}&= \frac{1}{2}\sum_{\alpha = 1}^{3}
\left (
z_{\alpha}\frac{\partial }{\partial  z_{\alpha}}
+ \frac{\partial }{\partial  z_{\alpha}} z_{\alpha}
\right )\\
X_{\mu}&= -\frac{1}{2} \sum_{\mu = 4}^{5} z_{\mu}^2,  &
\quad
Y_{\mu}&= \frac{1}{2} \sum_{ \mu= 4}^{5}
\frac{\partial ^2}{\partial z_{\mu}^2}, &
\quad
H_{\mu}&= \frac{1}{2}\sum_{\mu = 4}^{5}
\left (
z_{\mu}\frac{\partial }{\partial  z_{\mu}}
+ \frac{\partial }{\partial  z_{\mu}} z_{\mu}
\right )
\end{alignat*}
One verifies that each of these satisfies the relations for
$\mfr{sl}_2 (\R)$: $[H, X] = 2 X$, $[H, Y] = -2 Y$
$[X, Y] = H$. Evidently, $\mfr{l}'_{\alpha}$ =
$\mr{span}\{ H_{\alpha}, X_{\alpha}, Y_{\alpha}\}$
commutes with $\mfr{l}'_{\mu}$ =
$\mr{span}\{ H_{\mu}, X_{\mu}, Y_{\mu}\}$
and one checks that $\mfr{l}' = \mfr{l}'_{\alpha}\times \mfr{l}'_{\mu}$
commutes with the operators in $\omega (\mfr{k})$ where
$\mfr {k} = \mfr{o}(V) = \mfr{o}_{3,2}$ is the maximal
compact of $\mfr{o}(3, 2)$: one must compute that
all these operators $H_i, X_i, Y_i$ commute with
the operators $\omega (X_{\alpha\beta})$,  $\omega (X_{\mu\nu})$
that appear in \cite[Theorem 7.1, p. 155]{KM4}, which they do.
This gives explicit formulas for the dual reductive pair
$(\mfr{k}, \mfr{l}' = \mfr{l}'_{\alpha}\times \mfr{l}'_{\mu})$.
\par
The space $\mfr{l} ^{'(0, 2)}$ is spanned by the operators
 $Y_{\alpha}, \ Y_{\mu}$. Recalling that the first factor
 in $\mr{O}_3 \times \mr{O}_2$ acts by rotation in the variables
 $z_{\alpha}$ and the second factor acts on the variables
  $z_{\mu}$, it is first of all clear that there are no constant 
  or linear polynomials in $\mc{P}(\C^5)$ in the representation
  $\mbf{5}\otimes \mbf{1}$. It is also clear that the space of polynomials
  mentioned in the statement of the proposition are harmonic: they are
  annihilated by the operators $Y_{\alpha}, \ Y_{\mu}$, and do constitute
  a representation of type $\mbf{5}\otimes \mbf{1}$. This shows that 
  $\deg (\mbf{5}\otimes \mbf{1}) = 2$, and there cannot be any other 
  harmonic quadratic forms in the representation  $\mbf{5}\otimes \mbf{1}$. \\
\\
An Alternative Proof: It suffices to find the $SO(3) \times SO(2)$-modules in $\mathscr{I} _{\mbf{1}\otimes \mbf{1}}$ and $\mathscr{I} _{\mbf{5}\otimes \mbf{1}}$ that are of lowest degrees. These modules, according to Howe, are unique, homogeneous and harmonic. 
 The degree zero polynomials in $\mathcal P(\mathbb C^5)$, yield a trivial $SO(3) \times SO(2)$-module. Therefore, $\mathscr{H}(K)_{\mbf{1}\otimes \mbf{1}}$ is the one dimensional
$\C$-vector space spanned by $1\in \mathcal{P}(\C^5)$. The degree 
$1$ polynomials yield a standard representation
$\mbf{3} \otimes \mbf{1}$ of $SO(3) \times SO(2)$. 
They are of lowest degree in $\mathscr{I}_{\mbf{3} \otimes \mbf{1}}$.
Therefore  $\mathscr{H}(K)_{\mbf{3}\otimes \mbf{1}}$ is the three dimensional
$\C$-vector space spanned by the $z_{\alpha}\in \mathcal{P}(\C^5)$.
The space of degree 
$2$ polynomials is simply 
$$S^2(\mathbb C^3 \oplus \mathbb C^2) \cong S^2(\mathbb C^3) \oplus 
\mathbb C^3 \otimes \mathbb C^2 \oplus S^2(\mathbb C^2).$$
The first summand decomposes into 
$\mbf{5}\otimes \mbf{1} \oplus \mbf{1} \otimes \mbf{1}$. 
The $\mbf{5} \otimes \mbf{1}$ summand is spanned by
\[
\sum _{\alpha, \beta  = 1}^3 c_{\alpha \beta}z_{\alpha }z_{\beta},
\quad c_{\alpha \beta} = c_{\beta\alpha} \in \C,
\quad
\mr{with\ \ } \sum _{\alpha = 1}^3 c_{\alpha \alpha} = 0.
\]
Clearly, it
is of the lowest degree in $\mathscr{I} _{\mbf{5}\otimes \mbf{1}}$. It must be equal to $\mathscr{H}(K)_{\mbf{5}\otimes \mbf{1}}$.

\end{proof}

\subsection{}
\label{SS:fock5}
Kudla and Millson define   
\[
\boldsymbol{\varphi}^{+} = \sum _{\alpha, \beta = 1}^{3} 
z_{\alpha}z_{\beta}\, \omega _{\alpha 4}\wedge\omega _{\beta 5}
\]
This is an element of 
$\mr{Hom}_{K} (\wedge ^{1, 1}\mfr{p}, \ \mc{P})$. It is clear that 
$\boldsymbol{\varphi}^{+} $ induces a $K$-isomorphism from 
$\wedge ^{1, 1}\mfr{p},$ and $\mc{P} _{2,\alpha}$, the space of
quadratic forms in the variables $z_{\alpha}$, $\alpha = 1, 2, 3$. 
As representation spaces these are 
$\mbf{5}\otimes \mbf{1} \oplus \mbf{1}\otimes \mbf{1}$. Thus 
$\boldsymbol{\varphi}^{+} $ induces an isomorphism of isotypical spaces
\[
\begin{CD}
\left ( \wedge ^{1, 1}\mfr{p}\right ) _{\mbf{5}\otimes \mbf{1} }
 @>\sim >> 
 \left (\mc{P} _{2,\alpha} \right ) _{\mbf{5}\otimes \mbf{1} },
\end{CD}
\] 
which are irreducible representations of $K$, the right-hand side
being described in proposition \ref{P:harmonics}. It is 
easily seen that $d \boldsymbol {\varphi}^{+} = 0$, so that 
we have a class $[\kmphi] \in \mr{H}^{1, 1}(\mfr{g}, K; \ \mc{P})$.

\subsection{}
\label{SS:fock6}

Recall that the $\tilde{U(5)}$-finite vectors in the Fock model 
are the polynomials in $\mc{P}(\C^5)$. Given any $(\mfr{g}, \ K)$-module homomorphism
$\mc{P}(\C^5) \to \mc{A} (\Gamma \backslash G)$
to the space of automorphic forms, $G = \mr{SO}(3, 2)$, 
we get a map
\[
 \mr{H}^{1, 1}(\mfr{g}, K; \ \mc{P})\longrightarrow
  \mr{H}^{1, 1}(\mfr{g}, K; \ \mc{A} (\Gamma \backslash G)).
\]
For our purposes these are given by theta lifting. Let 
$\sigma \subset \mc{A} (\Gamma '\backslash G')$, 
$G' = \tilde{\mr{SL}}_2 (\R)$, be the space of holomorphic
cusp forms of weight $5/2$ belonging to an irreducible 
representation of $G'$. Given a linear functional 
$\Theta : \mc{P}(\C^5) \to \C$ with the
property that $\Theta (\omega (\gamma ', \gamma)\varphi) = 
\Theta (\varphi)$ for all $\varphi \in \mc{P}(\C ^5) $,
all $(\gamma ', \gamma) \in \Gamma ' \times \Gamma$,  
defining the theta kernel as 
$\theta _{\varphi}(g', g) = \Theta (\omega (g', g)\varphi)$, we 
define, for any $f \in \sigma$, 
\[
\theta _{\varphi}(f ) (g) = \int _{\Gamma '\backslash G'}
\theta _{\varphi}(g', g) f(g') dg', 
\]
and let $\theta_{\varphi}(\sigma) \subset  \mc{A} (\Gamma \backslash G)$
be the space spanned by the $\theta _{\varphi}(f )$. Note 
that, for any fixed $f$, the map $\varphi \mapsto \theta _{\varphi}(f )$
is a $(\mfr{g}, \ K)$-module homomorphism 
$\mathscr{S} (\C ^5)_{(K)} \to \mc{A} (\Gamma \backslash G)$.
Hence, each $f \in \sigma$ defines a map 
 
\[
\begin{CD}
 \mr{H}^{1, 1}(\mfr{g}, K; \ \mc{P}) @>\theta (f)>>
  \mr{H}^{1, 1}(\mfr{g}, K; \ \mc{A} (\Gamma \backslash G)).
\end{CD}  
\]
We define $\theta _{\kmphi} (f) := \theta (f)([\kmphi])$.
Finally note that the symbols $\theta (f)$, 
etc., defined here are ambiguous in that they 
depend on the initial choice of functional $\Theta$. 
In the theory of special cycles, functionals $\Theta$ are
constructed by summing over subsets of the form 
$x + L \subset \R ^5$ for rational vectors $x$ and lattices $L$. By varying $x$ and 
$L$ we obtain all the special cycles in that theory. This 
is best formulated in adelic language. 
It will always be assumed that our functionals
have this form. The more precise notation will be 
$\theta ^{x, L} _{\varphi} (f)$, etc., but we will follow Kudla and Millson
in simply writing $\theta _{\varphi} (f)$ when reference to the specific form of
the kernel is not needed. We will need:
\begin{lem}
\label{L:exchange}
We have $\theta _{Z \varphi} (f) = \theta _{\varphi} (Z^*f) $ for 
the involution $Z \to Z^*$ induced by the map
$g \to g^{-1}$ of $G'$. 
\end{lem}
Proof: The map $Z \rightarrow Z^*$ is given by
 $$ X_1 X_2 \ldots X_n \rightarrow (-1)^n X_n X_{n-1} \ldots X_1,$$
 and it is complex conjugate linear. Recall that $\theta: \varphi \in \mc P \rightarrow \theta_{\varphi} \in C^{\infty}(G^{\prime}, G)$ preserves the actions of $\mathscr{U}(\mathfrak g)$ and $\mathscr{U}(\mathfrak g^{\prime})$. $\forall \, Z \in \mathscr{U}(\mathfrak g^{\prime})$, we obtain
 \begin{align*}
\theta _{Z \varphi} (f) &= \int _{\Gamma '\backslash G'}
\theta _{Z \varphi}(g', g) f(g') dg'\\
&= \int _{\Gamma '\backslash G'}
Z(\theta _{\varphi})(g', g) f(g') dg' \\
&= \int _{\Gamma '\backslash G'}
\theta _{\varphi}(g', g) Z^*f(g') dg' \\
&= \theta _{\varphi} (Z^*f).
\end{align*}

\subsection{}
\label{SS:fock7}
According to Weissauer's theorems, any cohomology class
$\xi \in \mr{H}^{1, 1} (X_{\Gamma}, \C)_0$ in  the complement 
to the Lefschetz class occurs in 
$\mr{H}^{1, 1}(\mfr{g}, K; \theta (\sigma))$ where
$\theta (\sigma)\subset \mc{A}(\Gamma ' \backslash G')$ 
is a theta-lifting (see section \ref{SS:fock6}) 
belonging to a space $\sigma$ of holomorphic cusp forms of weight 
$5/2$ for
$G' = \tilde{\mr{SL}}_2(\R)$.
We can lift this to an element
$\xi \in 
\mr{Hom}_{K} (\wedge ^{1, 1}\mfr{p}, \ \theta (\sigma))$,
and without loss of generality we can assume that it factors as 
\[
\begin{CD}
\xi : \wedge ^{1, 1}\mfr{p}@>>>
\left ( \wedge ^{1, 1}\mfr{p}\right ) _{\mbf{5}\otimes \mbf{1}}
@>\xi _{0}>> [\theta (\sigma)]_{\mbf{5}\otimes \mbf{1}}
\end{CD}
\]
where the first arrow is projection onto the isotypical component, 
and the second arrow is an injection of $K$-modules.
These assertions follow from Vogan-Zuckerman theory \cite{VZ}:
any such class will factor through $\mr{H}^{1, 1}(\mfr{g}, K; A_{\mfr{q}})$ 
for an inclusion of the cohomological representation
$A_{\mfr{q}} \to \theta (\sigma)$. But 
the minimal $K$-type in $A_{\mfr{q}}$ is ${\mbf{5}\otimes \mbf{1}}$, 
with multiplicity one. According 
to the isomorphism in section \ref{SS:fock5} each 
$\varphi  \in \mc{P}(\C ^5)_{2, \alpha} =\mathscr{H}(K)_{\mbf{5}\otimes \mbf{1}} $ is 
equal to $\mbf{\varphi}^+ (v)$ for a unique 
$v \in \left (\wedge ^{1, 1}\mfr{p}\right ) _{\mbf{5}\otimes \mbf{1}}$.
Thus, for any $v \in \left (\wedge ^{1, 1}\mfr{p}\right ) _{\mbf{5}\otimes \mbf{1}}$ we can write, applying Howe's main result 
theorem \ref{P:harmonics},
\begin{align*}
\xi _0 (v) &= \sum \theta ^{x_j, L_j}_{\varphi _j} (f_j), \quad 
\varphi _j \in \mc{P}(\C ^5)_{\mbf{5}\otimes \mbf{1}}, 
\ f_j \in \sigma\\
&= \sum \theta ^{x_j, L_j} _{Z_j \varphi _{j}^{0}} (f_j), 
\quad Z_j \in \mathscr{U}(\mfr{g}'), \ 
\varphi _{j}^{0} \in \mc{H}(K) _{\mbf{5}\otimes \mbf{1}}\\
&=\sum \theta  ^{x_j, L_j}_{\varphi _{j}^{0}} (Z^*_j f_j), \quad
\mr{by\ lemma \ } \ref{L:exchange}\\
&=\sum \theta ^{x_j, L_j} _{\kmphi (v_j)} (Z^*_j f_j), \quad 
v_j \in \left (\wedge ^{1, 1}\mfr{p}\right ) _{\mbf{5}\otimes \mbf{1}}\\
&=\sum \theta  ^{x_j, L_j}_{\kmphi } (g_j) (v_j), \quad
g_j = Z^*_j f_j \in \sigma
\end{align*}
where the last line interprets 
$\theta _{\kmphi } (g_j) $ as a $K$-morphism
$\left (\wedge ^{1, 1}\mfr{p}\right ) _{\mbf{5}\otimes \mbf{1}}
\to \theta (\sigma )$. This calculation shows that the image of the 
$K$-morphism $\xi _0$ is contained in the sum of the images
of the $K$-morphisms $\theta _{\kmphi } (g_j) $ and since
the source of these maps is an irreducible $K$-representation, 
Schur's lemma implies that we must have 
$\xi _0 = \sum c_j \theta _{\kmphi } (g_j) $ for some 
$c_j \in \C$ or in other words, we have proved:
\begin{prop}
\label{P:main1}
Any cohomology class
$\xi \in \mr{H}^{1, 1} (X_{\Gamma})$ in  the complement 
of the Lefschetz class can be 
written as a finite linear combination of 
$\theta ^{x, L} _{\kmphi} (f)$ where $f \in \sigma$, and where 
$\sigma$ is an irreducible automorphic representation
belonging to the holomorphic cusp forms of weight $5/2$
for $\tilde{\mr{SL}}_2(\R)$, and $\kmphi$ is the 
special element of Kudla-Millson.  
\end{prop}

\section{Main Theorem}
\label{S:main}
We can now show that the Picard groups of Siegel modular threefolds 
are generated by special cycles by combining proposition
\ref{P:main1} with theorem \ref{T:KM2}. In the notations there, 
we have $p = 3$, $q=2$, $m = p+q = 5$, $n =1$, so that we are
in the range $n < m/4$. In this case, the special 
theta lifting is from the holomorphic cusp forms 
of weight $5/2$ for $\tilde{\mr{SL}}_2 (\Q)$ to 
harmonic $(1, 1)$-forms on the manifold $X_{\Gamma}$.

\begin{thm}
\label{T:main}
For any subgroup of finite index $\Gamma \subset \mr{Sp}_4 (\Z)$, $\mr{Pic} (X_{\Gamma})\otimes \Q$ is spanned by the 
classes of the special cycles. 
\end{thm}
\begin{proof}
We already know that 
$\mr{Pic} (X_{\Gamma})\otimes \C =  \mr{H}^{1, 1} (X_{\Gamma}, \ \C)$. 
We also know that 
$\mr{H}^{1, 1} (X_{\Gamma}, \ \C)
= \C \cdot\eta \oplus  \mr{H}^{1, 1} (X_{\Gamma})_0$, 
where $\eta$ is the Lefschetz class and 
$\mr{H}^{1, 1} (X_{\Gamma}, \ \C)_0$ is its canonical 
complement.  In fact, $\eta$ is in the span of the 
Humbert surfaces; this follows from Yamazaki's formula, 
\cite[Lemma 7]{tY}
\[
10 \eta  = 2[E] + N [D]
\]
for the principal congruence subgroup $\Gamma (N)$. Here $E$ is the divisor
which is the sum of the Humbert surfaces of discriminant 1, and 
$D$ is the sum of the boundary components. This formula holds on the
toroidal Igusa compactification of $X_{N}$. Since 
$X_N$ is the complement of the divisor $D$, this 
shows that a multiple of $\eta$ is in the span of the 
special cycles. 
   Thus it is enough to see that any class in the canonical
complement  $\mr{Pic} (X_{\Gamma})_0\otimes \Q$ is in the 
span of the special cycles. Evidently the special cycles
generate a subvector space. We know that 
$\mr{Pic} (X_{\Gamma})_0\otimes \C  =  \mr{H}^{1, 1} (X_{\Gamma})_0$ and by proposition \ref{P:main1}
these are generated by the special theta lifts
of Kudla-Millson $\theta ^{x, L} _{\kmphi} (f)$ 
for various lattices $L$. According to theorem 
\ref{T:KM2} these are in the span of special cycles. 
\end{proof}

\section{Special cycles}
\label{S:special}
\subsection{}
\label{SS:special1}
The theory of
special cycles has its origin in the discovery, by Hirzebruch and Zagier,
of special curves on Hilbert modular
surfaces and the connection of these with modular forms:
the intersection numbers of these special curves appear as Fourier
coefficients of modular forms.
This was vastly generalized by two groups:
Kudla and Millson,
\cite{KM1}, \cite{KM2}, \cite{KM3}, \cite{KM4}; and Tong and Wang, \cite{TY1},
 \cite{spW1}, \cite{spW2}. The symmetric
spaces in question are those associated to one of three classes
of groups: $O(p, q)$, $U(p, q)$ and $Sp(p, q)$. 
\subsection{}
\label{SS:special2}
We will follow the notations of the papers of Kudla-Millson. We explain
their results only in the orthogonal case.
\begin{alignat*}{2}
G &= O(p, q),   &\  &V = \mr{\ the\  standard \ module\ for \ } G\\
D &\cong \mr{SO}_0 (p, q)/(  \mr{SO} (p)\times \mr{SO} (q)), \ &
&\mr{the \ symmetric \ space \ for \ }G\\
(\ , \ )&: V \times V \to \R&\ &\mr{a\  symmetric \ bilinear\ form,\
             sgn } = (p, q) \\
L &\subset V & \ & \mr{a\ }\Z-\mr{lattice\ with \ }(L , L)\subset  \Z\\
\Gamma &\subset G &&
                     \mr{torsion\ free\ congruence\ subgroup}\\
G' &= \mr{Sp}(2n, \ \R )&\ \tilde{G}' &= \mr{Mp}(2n, \ \R )
\end{alignat*}
We assume that $\Gamma$ preserves the lattice $L$.
We can assume $p\ge q$. Here $n$ is an integer with
$1\le n \le p$.
\[
G_0 = \mr{SO}_0 (p, q) \mr{\ is\  the\  connected\
component \ of \ the \ identity\  in \ }G.
\]
It is also the set of elements
of spinorial norm 1. We can identify the symmetric space
$D$ with $\mr{Gr}^{-}_q (V)$, the subspace of the Grassmannian of $q$-planes
$Z$ such that $Z \mid (\ , \ )$ is negative definite.
Let $V_{\Q} = L \otimes _{\Z} \Q$. Define
$X_{\Gamma } = \Gamma \backslash D$, a smooth manifold of dimension
$a = pq$.
In case $q = 2$ this carries a canonical complex structure, and is
the set of $\C$-points of a Shimura variety; these are studied in
detail in \cite{K} in the anisotropic case.
The relevant case for us is 
when $(p, q) = (3, 2)$. Then there is an isogeny 
$\mr{Sp}_4 (\R) \sim \mr{SO}_0 (3, 2)$ and $D$ is 
isomorphic with the Siegel space of genus 2.
See sections \ref{S:iso}, and \ref{S:sym} for the dictionary to go between 
the symplectic and orthogonal viewpoints.

\subsection{}
\label{SS:special3}
We are now going to define certain cycles on $X_{\Gamma}$.
Let $U_{\Q} \subset V$ be an oriented subspace such that
$ (\ , \ )\mid U$ is nondegenerate. Here we use the convention
that suppression of an index such as $\Q$ means the $\R$-span
of the corresponding object; here $U = U_{\Q}\otimes _{\Q} \R$.

Then we have a
decomposition $V =  U \oplus U^{\perp} $.
We define
\[
D_U = \{Z \in D : Z = Z \cap  U +Z \cap  U^{\perp}
       \} \subset D
\]
and let $G_U$ be the stabilizer of $U$ in $G$ and
$G_{U}^0$ the connected component of the identity. Put
$\Gamma _U =  \Gamma \cap G_U$ and  $\Gamma _{U}^0 =  \Gamma \cap G_{U}^0$.
Define $C_U =\Gamma _{U}^0\backslash D_U $. The natural map
$\pi: C_U \to \Gamma \backslash D $ is proper, and thus
the pair $(C_U, \pi )$ is a locally finite singular cycle in $X_{\Gamma}$.
This will be orientable if $\Gamma$ is a subgroup of
$\mr{SO}_0 (p,q)$, which will always be the case in our examples,
since the $\Gamma$ we work with will be the images of corresponding
subgroups of $\mr{Sp} (4, \R )$, which is the spin covering of
$\mr{SO}_0 (3, 2)$.
\subsection{}
\label{SS:special4}
Given subspace $U$ as above we define an involution
$\tau = \tau _U$ via
\[
\tau =
\begin{cases}
+1 \mr {\ on\ } U\\
-1 \mr {\ on\ } U^{\perp}
\end{cases}.
\]
Then $G_U$ is the centralizer of $\tau$ and
$D_U = \{Z : \tau Z = Z  \}$. If $ (\ , \ )\mid U$ has signature
$(r, s)$ then $G_U  \cong \mr{O}(r, s) \times \mr{O}(p-r, q-s)$
and $D_U$ has codimension $p s + q r - 2 rs$ in $D$. Most important
for us is the case where the plane $U$ is positive definite.
In these cases
\[
D_U = \{Z \in D : Z \subset U^{\perp}
       \}
\]
and clearly
$n = \dim U \le p$. In general, when $q=2$, the cycles $C_U$ are algebraic cycles if
$U$ is positive definite; if $U$ is not positive-definite,
they are totally geodesic (singular) submanifolds
of $D$. From now on, we only consider positive definite $q$-planes.
\subsection{}
\label{SS:special5}
In our case, $(p, q)= (3, 2)$, and via the isomorphism 
$D \cong \mfr{H}_2$, 
$D_U$ is an embedded copy of
$\Sg{1}\times \Sg{1}$, resp.  $\Sg{1}$, resp. a point, 
corresponding to  $\dim U = 1$, resp. 2, resp. 3.
When  $\dim U = 1$, $D_U$ is a Humbert surface (see section \ref{S:sym}).
For a good discussion of Humbert surfaces, see \cite[Ch. IX]{vdG}. 

\subsection{}
\label{SS:special6}
We now define certain linear combinations of the cycles
$C_U$. Let $n$ be an integer in the range 1 to $p$. Let
$X = \{x_1, ..., x_n \}\subset V^n$; such an $X$ is called
an $n$-frame. Let
$(X, X)$ be the symmetric matrix given by
$(X, X)_{ij} = (x_i , x_j )$. We consider the orbit
$\mc{O} = G \cdot  X$ in $V^n$. We call the orbit {\it nonsingular} if
$\mr{rank} (X, X)= n$ and {\it nondegenerate} if
$\mr{rank} (X, X) = \dim _{\R} \mr{span} X$. The zero orbit is
nondegenerate but singular. It is known that the orbit
$\mc{O}$ is closed if and only if it is nondegenerate. If
$\mc{O}$ is a closed orbit, then by a theorem of Borel,
$\mc{O} \cap L^n$ consists of a finite number of $\Gamma$-orbits.
Letting ${Y_1 , ..., Y_l}$ be a set of representatives of these
orbits and $U_j = \mr{span} Y_j$ we define
\[
C_{\mc{O}} = \sum _{j=1}^{l} C_{U_{j}}.
\]
Given any symmetric $n\times n$ matrix $\beta$, define
\[
\mc{Q} _{\beta} = \{X \in  V^n : (X, X) = \beta \}.
\]
If $X \in \mc{Q} _{\beta}$ then $\mc{O} = G \cdot X \subset \mc{Q} _{\beta}$.
In case $\beta $ is positive definite, $ G$ acts transitively
on $\mc{Q} _{\beta}$, thus $\mc{O} = \mc{Q} _{\beta}$ and we may
write $C_{\beta }$ for $C_{\mc{O}}$. If $\beta$ is positive semidefinite
of rank $t < n$,   $\mc{Q} _{\beta}$ contains a unique closed orbit
defined by
\[
\mc{Q} _{\beta}^c = \{X \in  \mc{Q} _{\beta} :
              \dim \mr{span} X = t \}
\]
and it is known that $G$ acts transitively on $\mc{Q} _{\beta}^c$.
If $\beta $ is positive semidefinite, we let $C_{\beta}$ denote
$C_{\mc{O}}$ for this unique closed orbit.

In this paper, $(p, q) = (3, 2)$, $n = 1$, so that $\beta$ is a 
positive rational number. Classically one referred to $\beta$ as
the discriminant of the Humbert surface.
\subsection{}
\label{SS:special7}
As Kudla-Millson observe, these cycles are often zero for trivial
reasons. By their conventions on orientations the
frames $(x_1 , ..., x_n)$ and $(-x_1, x_2, ...x_n)$ could both occur
and give cancelling contributions. Therefore they fix
a $h \in L^n$ and an integer $N \ge 1$ such that
$\gamma \in \Gamma $ implies $\gamma \equiv 1$ mod $N$,
and replace $\mc{Q}_{\beta} \cap L^n$ in the above definition with
$\mc{Q}_{\beta} \cap (h +N L^n)$. Taking $\Gamma$-orbit representatives
in this set then defines cycles $C_{\beta, h, N}$ as above. Note that
Kudla-Millson simplify this notation to $C_{\beta}$ in all their subsequent
work.

\section{Theta correspondence }
\label{S:theta}
\subsection{}
\label{SS:theta1}
Given a dual reductive pair
$(G, G')$ in the sense of Howe
\cite{rH}, the theta correspondence is a mapping between automorphic
forms on $G$ and automorphic forms on $G'$, and conversely.
With $(G, G')= (\mr{O}(V), \mr{Sp}(2n, \R ))$ the general set-up
is as follows: Let $ \mc{S} (V^n )$ be the Schwartz space of
$C^{\infty} $ complex-valued functions
all of whose derivatives decrease rapidly to $0$ at infinity.
This carries a canonical action $\omega$ of
$G \times \tilde{G}'$, where $\tilde{G}' \to G'$ is the metaplectic
double cover. The first factor acts via:
\[
\omega (g) \varphi (X) = \alpha (g)\varphi (g^{-1}X).
\]
where $\alpha: G \to \C^*$ is the $n$th power of spinor norm.
The second factor acts via the Weil, or oscillator, representation.
Given any $\Gamma \times \tilde{\Gamma} '$-invariant distribution
($\tilde{\Gamma}'$ the inverse image of $\Gamma'$ in $\tilde{G}'$)
$\Theta :  \mc{S} (V^n )\to \C$, and any
$\varphi \in  \mc{S} (V^n )$ one can form the kernel
$\theta _{\varphi} (g, g') =\Theta(\omega (g)\omega (g')\varphi) $
then one defines
\begin{align*}
\theta _{\varphi}(f)(g') &= \int _{\Gamma \backslash G}
           f(g) \theta _{\varphi} (g, g')dg \\
\theta _{\varphi}(f')(g) &= \int _{\tilde{\Gamma} ' \backslash \tilde{G}'}
            f'(g')\theta _{\varphi} (g, g')dg'.
\end{align*}
The most important case for us will be the distributions
given by summing over a lattice in $V^n$. For instance, 
we can define, for $h \in L^n$, $N \in \Z$, 
\[
\Theta _{h, N} (\varphi )  =    \sum _{\substack{
                                 X \in L^n\\
                                 X \equiv h \mr{\ mod\ } N
                                  }} \varphi (X)
\]
From now on we assume our distribution has this form.
If $f$ (resp. $f'$) is a cusp form on $G$ (resp. $G'$) then
$\theta _{\varphi}(f)$ (resp. $\theta _{\varphi}(f')$ will be an
automorphic form on $G'$ (resp. $G$) provided that $\varphi$
is $K \times \tilde{K}'$ -finite where $K$ (resp. $\tilde{K}'$ is a maximal
compact subgroup of $G$ (resp. $G'$). 

\subsection{}
\label{SS:theta2}
In practice, $\varphi$ will transform
according to specific representations $\sigma$, $\sigma '$ of
$K$,  $\tilde{K}'$. These representations define homogeneous vector
bundles $E_{\sigma}$ (resp.$E_{\sigma'}$) on the symmetric space
$D$ (resp. $\Sg{n}$), and we may interpret $\theta _{\varphi}$ as
defining linear operators between spaces of sections:
\[
\Gamma (X_{\Gamma}, E_{\sigma}) \to
\Gamma (\Gamma  ' \backslash \Sg{n},  E_{\sigma'}),\qquad
\Gamma (\Gamma  ' \backslash \Sg{n},  E_{\sigma'})\to
\Gamma (X_{\Gamma}, E_{\sigma})
\]
The crucial case for us is when $\sigma $ defines the bundle
of differential forms of degree $nq$ on $D$ and $\sigma '$ defines
the line bundle $\mc{L}_m$ whose holomorphic sections are the Siegel cusp forms of
weight $m/2 = (p+q)/2$. 
It is a nontrivial fact that there exists a kernel 
$\theta _{\boldsymbol{\varphi} ^+}$ which gives rise to a linear map
\[
\Lambda : S_{m /2} (\Gamma ') \to \mbf{H}^{nq} (M_{\Gamma })
\]
where the left side above is the space of holomorphic cusp
forms of weight $m/2$ on a congruence subgroup
$\Gamma ' \subset \tilde{G}' = \mr{Mp}(2n, \R) $, and the right hand side
is the space of closed harmonic $nq$ forms on $X_{\Gamma }$. The element
$\boldsymbol{\varphi} ^+$ is then a Schwartz function with values in 
differential forms on $D$.

\subsection{}
\label{SS:theta3}
Kudla and Millson present this construction in the following way.
Let 
\[
\theta : \rm{H}_{ct} ^* (G, \mc{S } (V^n )) \to \coh{\ast}{X_{\Gamma }}{\C}
\]
where the left hand side above is continuous cohomology, be the
composite
\[
\begin{CD}
\rm{H}_{ct} ^* (G, \mc{S } (V^n )) @>\mr{restriction}>>
\rm{H} ^* (\Gamma , \mc{S } (V^n ))
@>\Theta >> \rm{H} ^* (\Gamma , \C) = \coh{\ast}{X_{\Gamma }}{\C} .
\end{CD}
\]
Here $\Theta : \mc{S } (V^n ) \to \C $ is any $\Gamma$-invariant distribution. By the van Est theorem, the continuous cohomology
is computed from the complex of $G$-invariant differential
forms
\[
\mr{C}^i = (A ^i (D) \otimes \mc{S}(V^n))^{G_0}
\]
so that we can identify a continuous cohomology class
with the class $[\varphi ]$ of a closed differential form
$\varphi$ on $D$ with
values in $\mc{S}(V^n)$.
Another way of understanding $\theta $ is the following:
If $\Theta$ is a $\Gamma $ -invariant distribution then
$g \mapsto \theta (\omega (g) \varphi )$ is in
$C^{\infty}(\Gamma \backslash G)$, for any
$[\varphi] \in  \rm{H}_{ct} ^* (G, \mc{S } (V^n ))$. Utilizing
well-known isomorphisms (see \cite{BW}), $\theta$ is the composite
\[
\begin{CD}
 \rm{H}_{ct} ^* (G, \mc{S } (V^n )) =
 \rm{H} ^* (\mfr{g}, K , \mc{S } (V^n ))
@>\varphi \mapsto \theta (\omega (g) \varphi ) >>
\rm{H} ^* (\mfr{g}, K ,C^{\infty}(\Gamma \backslash G) ) =
 \coh{\ast}{X_{\Gamma }}{\C}
\end{CD}
\]
Kudla and Millson construct a pairing
\[
((\ ,\ )) : \rm{H} ^{i}_{c} (X_{\Gamma}  ,\C ) \times
\rm{H}_{ct} ^{a-i} (G, \mc{S } (V^n ))\longrightarrow
C^{\infty}( \tilde{G}'), \qquad
a = \dim X_{\Gamma}
\]
via
\[
\theta _{\varphi} (\eta)(g') := ((\eta  ,\varphi ))(g') =
\int _{X_{\Gamma}} \eta \wedge \theta ((\omega (g') \varphi)
\]
where $\eta$ is represented by a closed compactly supported $i$-form
and $ \theta ((\omega (g') \varphi)$ by a closed $(a-i)$-form. One of the main
results of \cite{KM4} is that, with suitable restriction on
$\varphi$, the image of this is in the holomorphic sections in
$\Gamma (\mc{L}_m)$ the space
of Siegel modular forms of weight $m/2$ on $ \tilde{G}'$. The relevant
$\varphi$ define classes in an space they denote
\[
\rm{H}_{ct} ^{a-i} (G, \mc{S } (V^n ))^{\mfr{q}}_{\chi _{m}}
\]
whose precise definition can be found in the introduction of
\cite{KM4}. Thus they obtain a pairing:
\[
((\ ,\ )) : \rm{H} ^{i}_{c} (M_{\Gamma}  ,\C ) \times
\rm{H}_{ct} ^{a-i} (G, \mc{S } (V ))^{\mfr{q}}_{\chi _{m}}
\longrightarrow
\Gamma (\mc{L}_m).
\]
Kudla and Millson construct a canonical
element
$\boldsymbol{\varphi} ^+ = 
\boldsymbol{\varphi} ^{+}_{nq} \in
\rm{H}_{ct} ^{nq} (G, \mc{S } (V^n ))^{\mfr{q}}_{\chi _{m}}$
such that, for any  $\eta \in\rm{H} ^{i}_{c} (X_{\Gamma}  ,\C )  $,
the Fourier coefficients of the Siegel modular form 
$\theta _{\boldsymbol{\varphi}} (\eta) (\tau) $ are essentially given by the 
periods of $\eta$ over the special cycles. 
Recall that the Fourier expansion is given by
($\tau = u +iv$) 
\[
\theta _{\boldsymbol{\varphi} ^+}(u + i v) = \sum _{\beta \in \mathscr{L}}
a_{\beta }(v) \, \mr{exp}(2 \pi i\mr{Tr}\, (\beta u))
\]
the sum ranging over a lattice $\mathscr{L}$ in the space of symmetric 
matrices of size $n$ with $\Q$-coefficients.

They prove:
\begin{thm}
\label{T:KM4I}
\begin{itemize}
\item[(i)]
The induced pairing
\[
((\ ,\ )) : \rm{H} ^{i}_{c} (X_{\Gamma}  ,\C ) \times
\rm{H}_{ct} ^{a-i} (G, \mc{S } (V^n )))^{\mfr{q}}_{\chi _{m}}
\longrightarrow
\Gamma (\mc{L}_m)
\]
takes values in the holomorphic sections.
\item[(ii)]
If $\eta \in\rm{H} ^{i}_{c} (X_{\Gamma}  ,\C )  $ and
$\varphi \in \rm{H}_{ct} ^{a-i} (G, \mc{S } (V^n ))^{\mfr{q}}_{\chi _{m}}$
then all the Fourier coefficients $a_{\beta}$ of
$\theta _{\varphi} (\eta)(g')$ are zero except the positive semi-definite
ones. Suppose further that $\varphi $ takes values in $\mbf{S}(V^n)$,
the polynomial Fock space (see \cite[Intro.]{KM4} for the definition)
then these Fourier coefficients are expressible 
in terms of periods over the special cycles $C_{\beta}$. For the 
canonical class $\boldsymbol{\varphi} ^+ 
= \boldsymbol{\varphi} ^{+}_{nq} $, $i = nq$, and for positive definite 
$\beta$ one has
\[
a_{\beta }(\theta _{\boldsymbol{\varphi}^+} (\eta )) (v)
= \mr{e} ^{- 2 \pi \mr{Tr}(\beta v)}\int _{C_{\beta}} \eta
\]
\end{itemize}
\end{thm}

\subsection{}
\label{SS:theta4}
Recall the pairing
\[
((\ ,\ )) : \rm{H} ^{i}_{c} (M_{\Gamma}  ,\C ) \times
\rm{H}_{ct} ^{a-i} (G, \mc{S } (V ))^{\mfr{q}}_{\chi _{m}}
\longrightarrow
\Gamma (\mc{L}_m).
\]
The line bundle $\mc{L}_m$ is on the space $\Gamma ' \backslash \mfr{H}_n$
for some congruence subgroup $\Gamma '$. The holomorphic sections
of this bundle is the space  $S_{m/2}(\Gamma ')$ of Siegel 
cusp forms of weight $m/2$.
Let $f \in S_{m/2}(\Gamma ')$ be such a cusp form.
For any $\varphi \in
\rm{H}_{ct} ^{nq} (G, \mc{S } (V ))^{\mfr{q}}_{\chi _{m}}$
we get a linear functional
\[
\eta \mapsto
\int _{\Gamma ' \backslash G'}\theta _{\varphi}(\eta )(g')
\overline{f}(g') dg':
\mr{H}_c ^{(p-n)q} (X_{\Gamma},\, \C )\to \C ,\qquad
\theta _{\varphi}(\eta ) := ((\eta , \, \varphi ))
\]
which is essentially the Petersson inner product. 
By the perfect pairing given by Poincar\'e duality,
\[
\mr{H}^{nq} ( M_{\Gamma}) \times
\mr{H}^{(p-n)q} _c( M_{\Gamma})\longrightarrow \C
\]
this linear form is identified with a class
$\theta _{\varphi} (f )\in \coh{nq}{X_{\Gamma}}{\C}$. By construction:
\[
[\theta _{\varphi}(\eta ), f]:=
\int _{\Gamma ' \backslash G'}\theta _{\varphi}(\eta )
\overline{f} = \int _{X_{\Gamma}} \eta \wedge \theta _{\varphi} (f )
:=  ( \eta ,  \theta _{\varphi} (f )).
\]
The map $f \mapsto \mr{class \ of \ }
\theta _{\varphi} (f )$ is the theta lifting
$\Lambda _{\varphi}:  S_{m/2}(\Gamma ') \to \mr{H}^{nq}(M_{\Gamma})$.
We denote this simply by $\Lambda$ for the canonical 
$\varphi = \boldsymbol{\varphi} = \boldsymbol{\varphi} ^{+}_{nq}$.
\subsection{}
\label{SS:theta5}
Let
$\mbf{H}_{\theta}^{\perp}\subset \mr{H}_c ^{(p-n)q} (X_{\Gamma})$ be the
subspace of all classes of closed compactly supported
$(p-n)q$ forms that are orthogonal under the pairing
$[\ ,\ ]$ to the image of $\mr{S}_{m/2}(\Gamma ')$ under
$\Lambda$. Let
$\mbf{H}_{\mr{cycle}}^{\perp} \subset  \mr{H}_c ^{(p-n)q} (X_{\Gamma})$
be the space of all classes of closed compactly supported
$(p-n)q$ forms that have period $0$ over all the special cycles
$C_{\beta}$ with $\beta > 0$ (positive-definite).
Let $\mbf{H}_{\theta}, \mbf{H}_{\mr{cycle}}$ be the subspaces
of $\mr{H}^{nq} ( X_{\Gamma})$ defined by these by Poincar\'e
duality.
They prove:
\begin{thm}
\label{T:KM2}
\cite[Theorem 4.2]{KM3}
If $n < m/4$ then $\mbf{H}_{\theta} = \mbf{H}_{\mr{cycle}}$.
\end{thm}
Thus by Poincar\'e duality,
in the case of finite volume but noncompact quotients
the subspace of $\mr{H}^{nq} (X_{\Gamma})$ spanned by the duals
of the special cycles coincides with the space of theta lifts.
In the above theorem the special lifting kernel 
$ \boldsymbol{\varphi}^+ $ is used. It is important to realize that 
there is also an initial choice of a theta distribution of the type 
\[
\Theta _{h, N} (\varphi )  =    \sum _{\substack{
                                 X \in L^n\\
                                 X \equiv h \mr{\ mod\ } N
                                  }} \varphi (X).
\]
This choice appears both in the definition of the kernel
$\theta _{\boldsymbol{\varphi}^+}$ and in the definition of the 
special cycles $C_{\beta}$. This distribution is invariant under the
arithmetic subgroup $\Gamma \times \Gamma '$.
The above isomorphism should more properly
be written as $\mbf{H}_{\theta} ^{h, N} =\mbf{H}_{\mr{cycle}}^{h, N}$
In our application, $(p, q) = (3, 2)$, $n = 1$, $m = p+q = 5$, so we are in this 
stable range. Eventually we take the limit over all $(h, N)$.

\section{The isogeny $\mr{Sp}(4, \R) \sim \SO{3}{2} $ }
\label{S:iso}

\subsection{}
\label{SS:iso1}
Let $V_{\Q}= \Q^4$ with standard basis $e_i,\ i = 1, ..., 4$ and
let $\Psi$ be the alternating bilinear form
\[
\langle x, \,y  \rangle  = \langle x, \,y \rangle _{\Psi}
= \phantom{}^t x \Psi y,
\mr{\ where\ }
\Psi = \begin{pmatrix}
    0 & 1_2\\
    -1_2 & 0
    \end{pmatrix}
\]
The group of symplectic similitudes is defined as:
\[
\mr{GSp}(V_{\Q},\, \Psi) = \mr{GSp}(\Psi) = \mr{GSp}(4)
= \{g \in \mat{4}{4}{\Q}\mid\transp g \Psi g = \eta(g) \Psi
 \}
\]
where $\eta(g)\in \Q ^*$, called the multiplier of $g$. The map
$g \to \eta(g)$ is a homomorphism $\eta :\mr{GSp}(V_{\Q},\, \Psi)\to \Q^* $
whose kernel is the symplectic group $\sym {V_{\Q}}{\Psi}$. We can regard
$\mr{GSp}(\Psi)$ as defining a group scheme over $\Z$, whose points in any
ring $R$, denoted $\mr{GSp}(\Psi,\, R)$ or
 $\mr{GSp}(4,\, R)$, is defined by the same formulas as above, but with matrix
entries in $R$, with $\eta (g) \in R^*$ being a unit.
The same definitions can be given for any nondegenerate alternating
bilinear form, but recall that (over a field) these are all equivalent
by a change of coordinates.

\subsection{}
\label{SS:iso2}
Let $V_{\Z}$ be the free $\Z$-module with basis 
$e_1, e_2, e_3, e_4$. We have a symmetric bilinear form

\[
b : \wdg{2} V_{\Z} \times \wdg{2} V_{\Z}\to
\wdg{4} V_{\Z}\overset {\mr{det}}{\longrightarrow}\Z
\]
where the first arrow is wedge product and the second is the isomorphism
defined by $\det (e_1 \wedge e_2\wedge  e_3 \wedge e_4 ) = 1$.
Clearly, the natural action of $\mr{GL}(4)$ on $\wdg{2} V$
preserves this quadratic form up to scalars, namely the determinant
homomorphism. The subgroup $\mr{GSp}(4)$ stabilizes the
line spanned by
$\psi = e_1 \wedge e_3 + e_2 \wedge e_4$ and thus
we have a representation on the orthogonal complement relative to
$b$:
\[
 \alpha : \mr{GSp}(4) \to \mr{GO}_0(\psi ^{\perp},\,  b\mid\psi ^{\perp} )
:=  \mr{GO}_0 (b_{ \psi}),
\]
where the group  $\mr{GO}_0 (b_{ \psi} )$ is the connected
component of the group of orthogonal similitudes. Restricting this
to the subgroup $\mr{Sp}(4)$,
one knows:
\begin{prop}
\label{P:symplorth}
$\alpha$ induces an isomorphism
\[
\alpha : \mr{Sp}(4)/\{ \pm 1\} \to
 \mr{SO}_0 (b_{ \psi})
\]
\end{prop}
One checks that the 5-dimensional quadratic form $b_{\psi}$
has signature $(3, 2)$.

\subsection{}
\label{SS:iso3}
Let $\mr{Skew}(4, \Q)$ be the space of skew-symmetric matrices
of size $4$ with entries in $\Q$. There is a natural action 
of $\mr{GL}(4, \Q)$ on this space by $M \to g.M.\phantom{}^t g$.
Let 
\[
\Psi= \begin{pmatrix}
    0 & 1_2\\
    -1_2 & 0
    \end{pmatrix} \in \mr{Skew}(4, \Q). 
\]
The stabilizer of $\Psi$ for this action is $\mr{Sp}(4, \Q)$. 
One can check that the symmetric bilinear form on 
 $\mr{Skew}(4, \Q)$ defined by 
 \[
 b_0 (M, N) := \frac{1}{2}\mr{Tr}(M\Psi N\Psi )- 
 \frac{1}{4}\mr{Tr}(M\Psi)\mr{Tr}(N\Psi)
 \]
is invariant under all $M \to g.M.\phantom{}^t g$ for 
$g \in \mr{Sp}(4, \Q)$. It is also $\Z$-valued on 
$\mr{Skew}(4, \Z)$.
The space
\[
\Psi ^{\perp} := \{M \in \mr{Skew}(4, \Q) :
b_0(M, \Psi) = 0
              \}
\]
is $5$-dimensional and invariant under $\mr{Sp}(4, \Q)$.
We therefore obtain a morphism of algebraic groups
$\mr{Sp}(4) \to \mr O (\Psi ^{\perp})$. This necessarily lands in the 
connected component $\mr{SO}_0 (\Psi ^{\perp})$ since $\mr{Sp}(4) $
is connected. It is well-known that this is an isogeny 
with kernel $\pm 1$. The signature of the form on 
$\Psi ^{\perp}$ is $(3, 2)$.

\subsection{}
\label{SS:iso4}
Given
$\eta = \displaystyle{\sum _{i < j} r_{ij}e_i \wedge e_j}
\in \wdg{2}V_{\Q}$, we can associate the
skew-symmetric matrix
$R_{\eta} = R = (r_{ij})\in \mr{Skew}(4, \Q )$, 
$r_{ij}= -r_{ji}$.
This assignment sets up a $\mr{GL}(4, \Q)$-equivariant
isomorphism 
\[
\bigwedge ^2 V \cong \mr{Skew}(4),
\]
with the action of  $g \in \mr{GL}(4, \Q)$ given on the
skew-symmetric matrices as
\[
R \to g R \transp g.
\]
The form $\psi$ above maps to $\Psi$. Under this 
isomorphism the form $b$ in section \ref{SS:iso2} goes over
into the form $b_0$ of section \ref{SS:iso3}, 
ie., $b(\xi, \eta) = b_0 (R_{\xi}, R_{\eta})$. In 
coordinates, 
\[
b_0 (M, N) =   m_{12} n_{34} + m_{34}n_{12} + m_{23} n_{14}
+ m_{14} n_{23}- m_{24} n_{13} - m_{13}n_{24}.
\]
$\Psi ^{\perp}$ is defined by 
$m_{13}+m_{24} = 0 $, $\psi ^{\perp}$ is defined by
$r _{13}+r_{24} = 0$ and these bilinear forms restrict
to this subspace as:
\[
b_0 (M, N) =  2 m_{13} n_{13} + m_{12}n_{34} + m_{34} n_{12}
+ m_{14} n_{23}+ m_{23} n_{14}. 
\]

\subsection{}
\label{SS:iso5}
Here is a geometric
description of this. Consider, for any ring $R$,
the module $V_R = R^4$ with the standard alternating form $\Psi$.
Viewed this way, $V$ is
the affine space $\Aff{4}$ regarded as a scheme over the integers.
It is known that the
Grassmannian of 2-planes through the origin
${\bf Gr}_2 (V) = {\bf Gr} (2, 4) $, or
 equivalently the Grassmannian of lines in $\Proj{3}$,
${\bf Gr}_1 ({\bf P}(V)) =  {\bf PGr} (1, 3) $ is canonically imbedded
\[
{\bf Gr}_2 (V)  = {\bf Gr} (2, 4)\ \longrightarrow\
{\bf P} \left(\wdg{2} V  \right) = \Proj{5}
\]
whose image is a quadric. In coordinates $(x_1 , x_2 , x_3 , x_4 )$
on $\Aff{4}$, a plane $ h =x \wedge y$ is mapped to the vector of  2 by 2
minors
\[
p_{ij}(h)\ =\ \left\vert \begin{array}{cc}
                                     x_i & x_j \\
                                     y_i & y_j
                                        \end{array}
                      \right\vert
\]

The quadric is given by Pl\"ucker's relation
\[
Q(p)\ =\ p_{12} p_{34}\   - \ p_{13} p_{24}\   + \ p_{14} p_{23}\   =\ 0
\]
The alternating form $\Psi$ is given by the formula
$p_{13}  + p_{24}$, therefore a 2-plane through the origin
$\Aff{4}$ is isotropic for this skew-form
iff $p_{13} (h) + p_{24}(h) = 0$ . In other words, the variety
of isotropic lines in $\Proj{3}$ is given by these two equations.
We can think of this scheme ${\bf Gr}_2 (V, \Psi)$ as the
3-dimensional quadric in
$[ p_{12} ,  p_{13},  p_{14},  p_{23},  p_{24}]$ - space
given by
\[
q(p) \ = \  p_{13} ^2 \ +\ p_{12} p_{34}\   + \ p_{14} p_{23}\   =\ 0
\]
This variety is the dual symmetric space of the group
$G = \mr{Sp}(\Psi)$ and is thus isomorphic with
$G/Q$ where $Q$ is the parabolic subgroup stabilizing any
isotropic 2-plane.
In any ring where 2 is a unit this is easily seen to be equivalent
to the form $u_1 ^2 + u_2 ^2 + u_3 ^2 - u_4 ^2 - u_5 ^2 $, of
signature (3, 2).
\subsection{}
\label{SS:iso6}
Now ${\mr{GSp}} (V,\, \Psi)$ operates on $V$ fixing
$\Psi$ up to scalar multiple,
hence it acts on $\mathop{\wedge ^2}V$ fixing the subvarieties
$Q(p) = 0$, $p_{13}  + p_{24} = 0$, and hence fixing the variety $q(p) = 0$.
We therefore obtain a homomorphism from the symplectic similitudes to the
orthogonal similitudes of $q$, but as
 $\mr{GSp} (\Psi)$ is connected, we get a map
\[
{\mr{GSp}} (V, \Psi)\ \longrightarrow\ {\mr{GO}}_0 (q)
\]
which is an isogeny. The quadratic form $q$ is essentially
$b_{0, \psi}$ above. More precisely,
consider the wedge product
\[
\wdg{2} V \times \wdg{2} V\ \longrightarrow\
\wdg{4} V\ \simeq {\bf G}_a
\]
where the last isomorphism is gotten by the determinant.
Choose any basis $\{e_i\}$ of
$V_{\Z} = \Z ^{4}$. Then it is easily seen that
the quadratic form associated to this pairing is
$$2 Q(a) = 2( a_{12} a_{34}  - a_{13} a_{24}  + a _{14} a_{23})$$
Consider $\psi  =
 e_1 \wedge e_3 + e_2 \wedge e_4$.
Then the orthogonal
$U = \psi ^{\perp}$ carries the induced form
$2 (b^2 - ac - de)$ in the basis
\[
\{f_1 , f_2 , f_3 , f_4 , f_5 \}
=
\{ e_1 \wedge e_2 ,
 e_1 \wedge e_3 -  e_2 \wedge e_4  ,
-e_3 \wedge e_4,
 -e_1 \wedge e_4 ,  e_2 \wedge e_3\}
\]
For our purposes it is better to work with the form on the dual lattice
spanned by $\{f_1 , (f_2  )/2, f_3 , f_4 , f_5 \}$, or rather twice it,
given in this basis as
\[
\Delta (a, b, c, d, e)\ =\ b^2 - 4ac - 4de
\]
This is the quadratic form that appears in the theory of Humbert
surfaces.
\par

\section{Symmetric spaces and Humbert surfaces}
\label{S:sym}
\subsection{Symplectic viewpoint}
\label{SS:sview}
\subsubsection{}
\label{SSS:sview1}
The underlying analytic space $A$ of a complex abelian variety 
over $\C$ of dimension $n$ is a quotient $A = \C ^n/L$ where
$L \subset \C^n$ is a lattice.
Thus $L \otimes \R := V = \C^n$ has a complex structure $J$ 
($J^2 = -1$).
We have canonically $L = \mr{H}^1 (A, \Z)$, and 
$\mr{H}^s(A, \Z) = \mr{Hom}(\bigwedge ^s L, \Z)$ for all $s$. 
There is a Riemann form for $A$, or polarization, ie.,
an alternating bilinear form $\psi : L \times L \to \Z$ such that 
\begin{enumerate}
\item[1.]  $\psi (Ju, Jv) = \psi (u, v)$ for all $u, v \in V$. 
\item[2.] $\psi (v, Jv) > 0$ for all $0\ne v \in V$.
\end{enumerate}   
The first condition for a polarization is that
$\psi \in \mr{H}^2 (A, \Z)\cap \mr{H}^{1, 1}(A)$ in the 
Hodge structure on cohomology. Recall that 
the existence of a complex structure on $V$ is equivalent to the existence
of a Hodge structure with 
\[
V_{\C} = V\otimes \C = V^{-1, 0}\oplus V^{0, -1}
\]
where $ V^{-1, 0}$ (resp.  $ V^{0, -1}$) is the $+i$
(resp. $-i$) eigenspace for $J$. As is well known, 
the cohomology of $A$ then has a $\Z$-Hodge structure, 
with $\mr{H}^1 (A, \R)$ being the dual $\check{V} $. 
There is a canonical isomorphism
\[
A = \mr{H}_1(A, \Z)\backslash \mr{H}_1(A, \C)/F^0
=\mr{H}_1(A, \Z)\backslash  V^{-1, 0}. 
\]
The holomorphic tangent space at $0$ is canonically identified
$T_0 A  = \mr{H}_1(A, \C)/F^0 = V^{-1, 0}$. In coordinates
$z_1, ..., z_n$ we get a basis 
$\partial /\partial z_1, ..., \partial /\partial z_n$
and thus a dual basis $dz_1, ..., dz_n$ of 
$T^*_0  A= \mr{H}^0 (A, \Omega _{A/\C} )= \check{V}^{1, 0} = 
F^0 (\mr{H}^1(A, \C))$. 
The natural map 
$L = \mr{H}_1 (A, \Z) \to \mr{H}^1 (A, \R) \cong
\mr{Hom}_{\C}(\mr{H}^0 (A, \Omega _{A/\C} ), \C)
$ sends $\gamma $ to the functional $\omega \mapsto 
\int _{\gamma}\omega$, so that $L$ is the lattice of periods.  
\[
\begin{CD}
L &=&\mr{H}_1(A, \Z) \\
@VVV&&\\
V &=&\mr{H}_1(A, \R)@>\sim>> \mr{Hom}_{\C}(F^0\mr{H}^1, \C)\\
@VVV && @AAA\\
V_{\C} &=&\mr{H}_1(A, \C)@>>> \mr{Hom}_{\C}(\mr{H}^1(A, \C), \C)
\end{CD}
\]

\subsubsection{}
\label{SSS:sview2}
Let $\psi$ be a principal polarization, i.e., $\psi$ induces
an isomorphism $L \to \check{L} = \mr{Hom}(L, \Z)$. We may then find  
a basis $e_1, ..., e_{2n}$ of $L$ such that 
$\psi (e_i, e_j) = 0$ unless $|i-j| = n$, and 
$\psi (e_i, e_{n+i}) = 1$, for all $i = 1, ..., n$. 
It is also well-known 
that we may find a basis $\omega _1, ..., \omega _n$ for 
$\mr{H}^0(A, \Omega _{A/\C})$ such that the $n\times 2n$ period matrix 
$\int _{e_j}\omega _i$ has the shape $(\tau, 1_n)$ for some 
$\tau \in \mfr{H}_n$. Since we have 
$\omega _i  = \sum _{j = 1}^{2n}(\int _{e_j}\omega _i ) \check{e}_j$, 
once we fix the symplectic basis $e_i$ we can regard the row span 
$F_{\tau}$ of the matrix $(\tau, 1_n)$ as the subspace
\[
F_{\tau} = \mr{H}^0 (A, \Omega _{A/\C}) \subset \mr{H}^1(A, \C) = \C^n. 
\]
Thus there is an isomorphism
\[
\tau \mapsto F_{\tau}: 
\mfr{H}_n \simeq \mr{Gr}_n ^{+} (\check{V}_{\C}), 
\] 
where $\mr{Gr}_n ^{+} (\check{V}_{\C})$ is the Grassmannian 
of $n$-dimensional complex subspaces $F \subset \check{V}_{\C}$
such that 
\begin{itemize}
\item[a)] $F$ is $\check{\psi}$-isotropic, i.e.,
$\check{\psi}(x, y)= 0 $ for all $x, y\in F$.
\item[b)] $-i\check{\psi} (x, \bar{x})  >  0 $ for all $0\ne x\in F$,
where $\bar{x} $ denotes conjugation relative to $V_{\R}$.
\end{itemize}
Here, $\check{\psi}$ is the dual alternating form on $\check{V}$. 
Then $F_{\tau}$ is $F^0 \mr{H}^1(A_{\tau}, \C)$ for the Hodge 
structure on the principally polarized abelian variety
$A_{\tau} = \C^n /L_{\tau}$, where $L_{\tau}  = \Z^n + \Z^n \tau$. 
\subsubsection{}
\label{SSS:sview3}
Now we specialize the preceding to the case $n=2$.
Let
$T_{\Z} = \wdg{2}\check{V}_{\Z}$, 
Then $\tau \in \mfr{H}_2$ defines a Hodge structure of 
dimension $6$
\[
(T_{\C},T_{\Z}, F_{\tau})
\mr{\  of\  type\ } (2, 0)+(1, 1)+(0, 2).
\]
such that $T_{\Z} = \mr{H}^2 (A_{\tau}, \Z)$. 
Note that a real form $\eta \in \wdg{2}\check{V}_{\R}$ has type
$(1, 1)$ if and only if $\eta (u, v) = 0$ for all
$u, v \in F_{\tau}$. Since $F_{\tau }$ is generated by the
rows of $(\tau , 1)$ we see that a real form $\eta$ is of type
$(1, 1) $ if and only if (notation as in section \ref{SS:iso4})
\[
(\tau , 1) R_{\eta}\begin{pmatrix}\tau \\ 1 \end{pmatrix} = 0.
\]
Since the N\'eron-Severi group of $A_{\tau }$ is isomorphic with
$ \mr{H}^2 ( A_{\tau }, \, \Z )\cap \mr{H}^{1, 1}(A_{\tau})$ we see

\[
\mr{NS}( A_{\tau }) \cong
\{
R \in \mr{Skew}(4, \Z ) :
(\tau , 1) R \begin{pmatrix}\tau \\ 1 \end{pmatrix} = 0
\}.
\]
Now let $\mr{Skew}(4,  R)_0:= \Psi ^{\perp} \subset \mr{Skew}(4, R)$, for any
commutative ring $R\subset \C $, where the orthogonal is taken with respect to 
the inner product in section \ref{SS:iso3}.  

We define, for any $ X \in \mr{Skew} (4, \R )_0$
\[
\Sg{2} \supset \mfr{H}_X := \{\tau \in \Sg{2} :
(\tau , 1) X\begin{pmatrix}\tau \\ 1 \end{pmatrix} = 0 \}.
\]
Clearly $\mfr{H}_X =\mfr{H}_{t X} $ for any $t \in \R ^*$.
It can be shown that this set is nonempty if and only
if $\Delta (X) > 0$, in which case it is isomorphic with
$\Sg{1}\times \Sg{1}$.
When the entries of $X$ are integers without common divisor,
we call this the Humbert surface associated to $X$, provided it is nonempty,
and we call
$\Delta (X)$ the discriminant of the Humbert surface. It is a
positive integer congruent to 0 or 1 modulo 4.
\par
Dually we define, for any $\tau \in \Sg{2}$, and any 
subring
$A \subset \C$, 
\[
T^{1, 1}_{A} (\tau )  =  \{ X\in \mr{Skew}(4, \, A )  :
(\tau , 1) X\begin{pmatrix}\tau \\ 1 \end{pmatrix} = 0 \}.
\]

\subsection{Orthogonal viewpoint}
\label{SS:orthview}
\subsubsection{}
\label{SSS:orthview1}
The symmetric space for the group $\mr{SO}(3, 2)$ is
\[
\mr{Pos}_{3, 2}=
\{
Z \in \mr{M}_{5, 3}(\R ) : \transp MI_{3, 2} M > 0
\},
\quad
I_{3, 2} =
\begin{pmatrix}
1_3 & 0\\
0 & -1_2
\end{pmatrix}
\]
More generally, replacing the split form $I_{3, 2}$ by any
real symmetric $B$ of signature $(3, 2)$, the condition is that
the real symmetric 3 by 3 matrix $\transp M B M $ is positive definite.
More intrinsically, fix a 5-dimensional real vector space
$T_{\R}$ with a symmetric bilinear form $b$ of signature
$(3, 2)$. Then the symmetric space is the open subset of the Grassmannian
of 3-planes in $T_{\R}$:
\[
\mr{Gr}^{+}_3 (T_{\R}) :=
\{
U \subset T_{\R} : \dim U = 3,\  b \mid U > 0
\}
\]
By choosing an orthonormal basis we can 
identify $T_{\R} = \R ^ 5$, and 
$\mr{Pos}_{3, 2}$ with $\mr{Gr}^{+}_3 (T_{\R})$ by 
assigning to $M$ the subspace of $\R ^ 5$ spanned by the rows
of $M$.
\subsubsection{}
\label{SSS:orthview2}
For any commutative ring $R$ let $V_R = R^4$ with 
basis $\underline{e} = \{e_1, e_2, e_3, e_4\}$. Let
$W = \mr{Hom}(V, \mbf{G}_a)$ be the dual with 
dual coordinates $\check{e}_i$. Let
$\psi \in \wedge ^2 (W) = \mr{Hom}(\wedge ^2(V), \mbf{G}_a)$
be the alternating bilinear form whose matrix in the 
coordinates $\underline{e} $ is 
\[
\Psi = \begin{pmatrix}
0 & 1 _2\\
-1_2 & 0
\end{pmatrix}.
\]
Each point $\tau \in \mfr{H}_2$ determines a Hodge structure 
of type $(1, 0), (0, 1)$ on $W$, ie.,
\[
W^{\tau} = \left (
W_{\C}= W^{1, 0}\oplus W^{0, 1}\supset W_{\R} \supset
          W_{\Z}.
            \right )
\]
Concretely, $W^{0, 1} = F_{\tau}^0 (W_{\C})$ is the span 
of the rows of the  matrix $(\tau, 1_2)$. This is the canonical 
Hodge structure  on $\mr{H}^1(A_{\tau})$ for the 2-dimensional 
abelian variety $A_{\tau} = \C ^2 /\Z^2 \tau + \Z ^2$. The 
form $\psi$ can be identified to an element of 
$\mr{H}^2 (A_{\tau}, \Z)\cap \mr{H}^{1, 1}(A_{\tau})$, and is 
a principal polarization. 

\subsubsection{}
\label{SSS:orthview3}
Each $\tau\in \mfr{H}_2$ thus gives a Hodge structure on any 
tensor space of $W$. In particular, consider 
$\wedge ^2 (W)$. Since $\psi \in \wedge ^2 (W)$ is of 
type $(1, 1)$ for this Hodge structure, and since 
the bilinear form of section \ref{SS:iso2} is a morphism 
of Hodge structures, the orthogonal space
$T := \psi ^{\perp}\subset \wedge ^2 (W)$ carries a 
$\Z$-Hodge structure. We have seen that the bilinear form
denoted $b_{\psi}$ on $T_{\R}$ has signature $(3, 2)$. Thus, 
any $\tau\in \mfr{H}_2$ gives a Hodge structure
\[
T^{\tau} = \left (
T_{\C} = T^{2, 0}\oplus T^{1, 1}\oplus T^{0, 2}\supset
T_{\R}\supset T_{\Z}
           \right )
\]
The space $T^{1, 1}$ is the complexification of a 
real $3$-dimensional subspace $Z_{\tau}\subset T_{\R}$ and it is known that 
$b_{\psi} \mid Z_{\tau}$ is positive-definite.

\begin{prop}
\label{P:somoduli}
The map $\tau \to Z_{\tau}$
sets up an isomorphism
\[
 \Sg{2}  \simeq \mr{Gr} ^{+}_3 (T_{\R})
\]
This map is equivariant, 
via the isogeny $\rho : \mr{Sp}_4 (\R) \to 
\mr{SO}_0 (b_{\psi}) = \mr{SO}_0 (3, 2)$:
$Z_{g\tau} = \rho (g) Z_{\tau}$. 
\end{prop}
Note that we have a canonical equivariant isomorphism
$\mr{Gr} ^{+}_3 (T_{\R}) =\mr{Gr} ^{-}_2 (T_{\R})$
where the right-hand side is the Grassmannian on
$2$-planes $Z'$ in $T_{\R}$ such that $b_{\psi} \mid Z'$ is negative-definite:
let $Z' = Z ^{\perp}$.

\subsubsection{}
\label{SSS:orthview4}
In section \ref{S:special} the letter $V$ represents a $5$-dimensional 
real vector space with a quadratic form 
$(\phantom{x}, \phantom{y})$ of signature $(3,2)$, which corresponds
to $T_{\R}$ here, and the lattice $L$ in that section corresponds to 
$T_{\Z}$. Also in the notations of that section, $D= \mr{Gr} ^{-}_2 (T_{\R})$. For any $x \in L = T_{\Z}$ with 
$(x, x) > 0$, let 
$U = \R .x \subset V = T_{\R}$. Under the natural identification
$ x \mapsto X: T_{\Z} = \mr{Skew}(4, \Z)_0$ (see sections \ref{SS:iso3}, 
\ref{SSS:sview3}), the Humbert surface $\mfr{H}_X$ maps to 
the special locus $D_U$ of section \ref{S:special}. This is because:
\begin{align*}
Z' \in D_U &\Longleftrightarrow Z'\subset U^{\perp}\\
            &\Longleftrightarrow U\subset(Z')^{\perp}:=Z\\
	    &\Longleftrightarrow U\subset Z_{\tau}
	    \mr{\ for\ a \ unique\ }\tau \in \mfr{H}_2\\
	    &\Longleftrightarrow x \mr{\ is\ of\ type\ } (1, 1)
	    \mr{\ in\  the\  Hodge\ structure\ } T^{\tau}\\ 
	    &\Longleftrightarrow 
	    (\tau , 1) X\begin{pmatrix}\tau \\ 1 \end{pmatrix} = 0\\
	     &\Longleftrightarrow \tau \in \mfr{H}_X.
\end{align*}

\section{Cohomological unitary representations}
\label{S:cohrep}
This section records the basic facts relevant to us
from Vogan-Zuckerman theory. These results are well-known
in that they have appeared in print on multiple occasions, 
but with no proofs. We do not provide complete proofs 
either, but at least some more detail. It is convenient
for us to work with the orthogonal as opposed to the 
symplectic viewpoint. So 
$\mfr{g} = \mfr{sp}_4 (\R) = \mfr{so}(3, 2)$. 

\subsection{}
\label{SS:cohrep1}
\[
 \mfr{so}(3, 2) = \left \{
 \begin{pmatrix}
A &B\\
C&D
\end{pmatrix} :
A = -\phantom{}^t A \in \mr{M}_{3, 3}(\R),\, 
D = -\phantom{}^t D \in \mr{M}_{2, 2}(\R),\,
B \in \mr{M}_{3, 2}(\R), \, C = \phantom{}^t B
 \right\}
\]
The Cartan decomposition $\mfr{g}= \mfr{k}\oplus \mfr{p}$
has $\mfr{k} = \{ B = C = 0\} = \mfr{so}(3) \times \mfr{so}(2)$, 
and 
\[
\mfr{p} = \{ A = D = 0\} \cong \mr{M}_{3, 2}(\R)\quad
\mr{via\ }
 \begin{pmatrix}
0&B\\
\phantom{}^t B&0
\end{pmatrix}\mapsto
B. 
\]
The action of $(A, D)\in \mr{O}(3)\times \mr{O}(2)$ by
conjugation on $\mfr{p}= \mfr{g}/\mfr{k} = \mr{M}_{3, 2}(\R)$ is 
\[
X \mapsto AXD^{-1}. 
\]
The complex structure on $\mfr{p}= \mr{M}_{3, 2}(\R)$ is
given by right multiplication by 
\[
J = \begin{pmatrix}
0&1\\
-1&0
\end{pmatrix}\quad
\mr{so\ that\ } \mfr{p}_{\C} = \mfr{p}^{+}\oplus
\mfr{p}^{-}, \ 
\mfr{p}^{\pm} = \pm i \mr{\ eigenspace\ of\ }J.  
\]

\subsection{}
\label{SS:cohrep2}
A compact maximal torus for $\mfr{g}$ is
\[
\mfr{t} = \left \{
 \begin{pmatrix}
x_1 J &\begin{matrix} 0\\0 \end{matrix} & 
\begin{matrix} 0&0\\ 0&0 \end{matrix}\\
\begin{matrix} 0&0\end{matrix}& 0 &
\begin{matrix} 0&0\end{matrix}\\
\begin{matrix} 0&0\\ 0&0 \end{matrix}
& \begin{matrix} 0\\0 \end{matrix}&
x_2 J
\end{pmatrix} :x_1, \ x_2 \in \R
 \right\}  = 
 \{
 [x_1,\ x_2] 
 \}.
\] 
The roots are
\[
\Delta (\mfr{g}_{\C}, \mfr{t}) = 
\{
\pm \alpha,\  \pm \beta, \ \pm \alpha \ \pm \beta 
\}
\]
where $\alpha ([x_1, x_2]) = i\, x_1$, 
$\beta ([x_1, x_2]) = i\, x_2 $.
We have 
\begin{align*}
\mfr{k}_{\C}  &= \mfr{t}_{\C} \oplus \mfr{g}_{\alpha }\oplus 
        \mfr{g}_{-\alpha }               \\
\mfr{p}^{+} &=  \mfr{g}_{\beta }\oplus 
 \oplus \mfr{g}_{\alpha +\beta}\oplus  \oplus \mfr{g}_{-\alpha +\beta} \\
 \mfr{p}^{-} &=  \mfr{g}_{-\beta }\oplus 
 \oplus \mfr{g}_{\alpha -\beta}\oplus  \oplus \mfr{g}_{-\alpha -\beta} 
\end{align*}
Writing $\mfr{g}_{\gamma} = \C X _{\gamma}$, we have:
\begin{align*}
X_{\alpha} = 
\begin{pmatrix}
 0&0&1&0&0\\
 0&0&i&0&0\\
  -1&-i&0&0&0\\
   0&0&0&0&0\\
    0&0&0&0&0\\
\end{pmatrix}
\quad\quad
&
X_{\beta} =   \begin{pmatrix}
 0&0&0&0&0\\
 0&0&0&0&0\\
  0&0&0&i&-1\\
   0&0&i&0&0\\
    0&0&-1&0&0\\
\end{pmatrix}\\
\\
X_{\alpha +\beta} = 
\begin{pmatrix}
 0&0&0&-i&1\\
 0&0&0&1&i\\
  0&0&0&0&0\\
   -i&1&0&0&0\\
    1&i&0&0&0\\
\end{pmatrix}
\quad\quad
&
X_{\alpha -\beta} =   \begin{pmatrix}
 0&0&0&i&1\\
 0&0&0&-1&i\\
  0&0&0&0&0\\
   i&-1&0&0&0\\
    1&i&0&0&0\\
\end{pmatrix}
\end{align*}
The map $i \mapsto -i$ sends root/spaces
$\gamma \to -\gamma$.

\subsection{}
\label{SS:cohrep3}
The unitary representations with nonzero $(\mfr{g}, K)$-cohomology
are of the form $A_{\mfr{q}}$ for $\theta$-stable parabolic
subalgebras $\mfr{q}\subset \mfr{g}_{\C}$ (more generally
$A_{\mfr{q}}(\lambda)$ for coefficients in a local system).
These parabolics can be taken up to $K$-conjugation. Each such
$\mfr{q}$ can be constructed by choosing a $x \in i \mfr{t}$
and defining 
\begin{align*}
\mfr{q}&= \mr{sum\ of\ the\ nonnegative\ eigenspaces\ of\ ad}(x).\\
\mfr{l}&= \mr{sum\ of\ the\ zero\ eigenspaces\ of\ ad}(x).\\
\mfr{u}&= \mr{sum\ of\ the\ positive\ eigenspaces\ of\ ad}(x).\\
\end{align*}
Then we have a Levi decomposition $\mfr{q} = \mfr{l}\oplus\mfr{u}$. 
Also, if $R^{\pm} = \dim (\mfr{u}\cap \mfr{p}^{\pm})$ and
$p-R^+ = q-R^- = j \ge 0$
\[
\mr{H}^{p, q}(\mfr{g}, K; \C) = 
\mr{Hom}_{L\cap K}(\wedge ^{2j}(\mfr{l}\cap \mfr{p}), \C).
\]
If $\mfr{f}\subset \mfr{q}$ is any subspace stable under
$\mr{ad}(\mfr{t)}$, we let $\rho(\mfr{f})$ be as usual half
the sum of the roots of $\mfr{t}$ occurring in $\mfr{f}$. 
Then for a $\theta$-stable parabolic $\mfr{q}$ it is known 
that if a representation of $\mfr{k}$ with highest weight 
$\delta \in \Delta ^+ (\mfr{k}, \mfr{t})$ occurs in 
$A_{\mfr{q}}$, then 
\[
\delta = 2 \rho (\mfr{u}\cap \mfr{p}) +
\sum_{\gamma \in \Delta (\mfr{u}\cap\mfr{p})}n_{\gamma} \gamma ,
\]
for integers $n_{\gamma}\ge 0$, and the representation of 
$K$ with highest weight $2 \rho (\mfr{u}\cap \mfr{p}) $
exists and occurs in $A_{\mfr{q}}$ ($K$ is the connected 
Lie group with Lie algebra $\mfr{k}$).
\subsection{}
\label{SS:cohrep5}
The cohomological representations
are gotten by choosing, respectively $x \in i\, \mfr{t}$ as follows
(see \cite[pp. 91-92]{HL}): 
\begin{align*}
& x = [0, 0], \ \ L \cong \mr{SO}_0 (3, 2), \ \ 
 \mr{nonzero\  in \ bidegrees\ }
(j, j)\mr{\ for\ } 0\le j \le 3. \\
& x = [-i x_1, 0], \ x_1 > 0,\ \  L \cong S^1 \times\mr{SO}_0 (1, 2), \ \ 
\mr{nonzero\  in \ bidegrees\ }
(j, j)\mr{\ for\ } 1\le j \le 2. \\
& x = [-i | x_2 |, i x_2], \ x_2 \ne  0,\ \  L \cong \mr{U} (1, 1), \ \ 
\mr{nz\  in \ bideg\ }
(2, 0), (3, 1), \mr{\ if\ } x_2 < 0; (0, 2),(1, 3) \mr{\ if\ } x_2 > 0.\\ 
& x = [-i x_1, i x_2], \ x_1 >  |x_2|\ne 0,\ \  
L \cong S^1 \times \mr{U} (0, 1), \ \ 
\mr{nz\  in \ bideg\ }
(2, 1), \mr{\ if\ } x_2 < 0; (1, 2), \mr{\ if\ } x_2 > 0. \\
& x = [-i x_1, i x_2], \  |x_2|> x_1 >  0,\ \  
L \cong S^1 \times \mr{U} (0, 1), \ \ 
\mr{nz\  in \ bideg\ }
(3, 0), \mr{\ if\ } x_2 < 0; (0, 3), \mr{\ if\ } x_2 > 0.
\end{align*}
A complete list even with nontrivial local coefficients
can be found in \cite{rT}. 
\subsection{}
\label{SS:cohrep4}
If we choose $x = [-i\, x_1, 0] \in i\, \mfr{t}$ with 
$x_1 > 0$, we find 
\[
\mfr{l} = \mfr{t}_{\C} \oplus \mfr{g}_{\beta}\oplus \mfr{g}_{-\beta},
\quad 
\mfr{u} = \mfr{g}_{\alpha}\oplus \mfr{g}_{-\beta+\alpha}\oplus \mfr{g}_{\beta +\alpha}, \quad
R^{\pm} = \pm 1.
\]
The $K = \mr{SO}(3)\times \mr{SO}(2)$-representation 
$\mu (\mfr{q})$ with highest weight 
$2 \rho (\mfr{u}\cap \mfr{p}) = 2\alpha$ is the tensor 
product $\mbf{5}\otimes \mbf{1}$ of the irreducible 5-dimensional
representation of $\mr{SO}(3)$  with the trivial representation 
of $\mr{SO}(2)$. Since this representation occurs 
with multiplicity one in $\wedge ^{1, 1} \mfr{p}$ we see that 
\[
\mr{H}^{1, 1}(\mfr{g}, K; A_{\mfr{q}}) = 
\mr{Hom}_K (\wedge ^{1, 1} \mfr{p},\  \mu (\mfr{q})) = 
\mr{Hom}_{K\cap L} (\wedge ^{0} (\mfr{l}\cap \mfr{p}),\ \C )) 
= \C
\] 
is one-dimensional, and also that $\mu (\mfr{q})$ occurs 
with multiplicity one in $A_{\mfr{q}}$.

\subsection{}
\label{SS:cohrep6}
Let $\pi_{\frac{5}{2}}$ be the discrete series representation of $\tilde{SL}_2(\mathbb R)$ of lowest weight $\frac{5}{2}$.
Now we must show that the representation $\theta(\pi_{\frac{5}{2}})$ has the following properties:
\begin{itemize}
\item[i.]
The minimal $K$-type is the $5$ dimensional representation $\mbf{5}\otimes \mbf{1}$;
\item[ii.]
The infinitesimal character is equal to the infinitesimal character of the trivial representation, $(\frac{3}{2}, \frac{1}{2})$.
\end{itemize}
Then by the theorem of Vogan-Zuckerman, $\theta(\pi_{\frac{5}{2}})$ is the $A_{\mfr{q}}$ for which $\mr{H}^{1, 1}(\mfr{g}, K; A_{\mfr{q}})$ is nonvanishing. This is a direct consequence of a theorem of Jian-Shu Li (~\cite{Li0}).\\
\\
As a matter of fact, this can be seen by the following observations. First, back to Prop. ~\ref{P:harmonics}, the compact group $\tilde U(1) \subseteq \tilde{SL}_2(\mathbb R)$ acts on the constant functions by $x^{\frac{1}{2}}, (x \in \tilde{U}(1))$. It acts on the first three variables by $x^{\frac{3}{2}}, (x \in \tilde U(1))$. Consequently, it acts on polynomials of degree $2$ on the first $3$ variables by $x^{\frac{5}{2}}, (x \in \tilde U(1))$.  In addition, the polynomials of degree $2$ on the first $3$ variables give the first occurence of representations of weight $\frac{5}{2}$ for $\tilde U(1)$ in $\mathcal P$.  So the constituent $\mbf{5}\otimes \mbf{1}$ consists of the joint harmonics for $\tilde U(1)$ and for $ O(3) \times  O(2)$. By Howe's theorem, $\theta(\pi_{\frac{5}{2}})$ contains a unique $O(3) \times O(2)$-type $\mbf{5}\otimes \mbf{1}$. So  $\mbf{5}\otimes \mbf{1}$ is the $O(3) \times O(2)$-type that is of minimal degree $\theta(\pi_{\frac{5}{2}})$ in the sense of Howe. As we have seen from the proof of Prop. ~\ref{P:harmonics}, the $O(3) \times O(2)$-types of smaller degrees, $\mbf{3}\otimes \mbf{1}$ or $\mbf{1}\otimes \mbf{1}$, must not occur in $\theta(\pi_{\frac{5}{2}})$. Therefore,  $\mbf{5}\otimes \mbf{1}$ is also the minimal $O(3) \times O(2)$-type of $\theta(\pi_{\frac{5}{2}})$ in the sense of  Vogan.\\ 
\\
Secondly, the infinitesimal character of $\pi_{\frac{5}{2}}$ is $(\frac{3}{2})$, under the Harish-Chandra homomorphism. By a theorem of Przebinda ~\cite{PZ}, the infinitesimal character of $\theta(\pi_{\frac{5}{2}})$ can be obtained from $(\frac{3}{2})$ by adding an entry $\frac{1}{2}$. So
the infinitesimal character of $\theta(\pi_{\frac{5}{2}})$ is exactly $(\frac{3}{2}, \frac{1}{2})$.

\bibliographystyle{amsplain}

 \end{document}